\documentclass[11pt,a4paper]{article}

\usepackage[a4paper,hcentering,vcentering]{geometry}
\usepackage[utf8]{inputenc}
\usepackage{nameref}
\usepackage[right]{lineno}
\usepackage{graphicx}
\usepackage{tikz-cd}

\usepackage{amsthm,amsmath,amsfonts,mathtools}
\newtheorem{theorem}{Theorem}[section]
\newtheorem{Pro}{Proposition}[section]

\newtheorem{Lem}{Lemma}[section]
\newtheorem{Col}{Corollary}[section]
\newtheorem{MainThm}{Theorem}

\newtheorem{Def}{Definition}[section]
\theoremstyle{definition}

\newtheorem{Exp}{Exemple}

\theoremstyle{remark}
\newtheorem{Rmk}{Remark}

\newcommand{\QED}{\hfill \qedsymbol{}}
\newcommand{\x}{\ensuremath{\mathbf{x}}}
\newcommand{\y}{\ensuremath{\mathbf{y}}}
\newcommand{\z}{\ensuremath{\mathbf{z}}}
\newcommand{\Se}{\ensuremath{(\mathbb{S}^2_g)^n_{\epsilon}}}
\newcommand{\Sd}{\ensuremath{(\mathbb{S}^2_g)^n}_0}
\newcommand{\norm}[1]{\lVert {#1} \rVert}

\begin{document}

\title{The $N$-vortex Problem on a Riemann Sphere} 

\author{\small Qun WANG \\  \footnotesize School of Mathematics and Statistics, Henan University, Kaifeng 475000, China \\ \footnotesize Department of Mathematics, University of Toronto, Toronto M5S 2E4, Canada \\ \url{wangqun927@gmail.com}}

\date{}
\maketitle
\noindent
\textbf{Abstract:} 
This article investigates the dynamical behaviours of the $n$-vortex problem with vorticity $\mathbf{\Gamma}$ on a Riemann sphere $\mathbb{S}^2$ equipped with an arbitrary metric $g$. From perspectives of Riemannian geometry and symplectic geometry, we study the invariant orbits and prove that with some constraints on vorticity $\mathbf{\Gamma}$, the $n$-vortex problem possesses finitely many fixed points and infinitely many periodic orbits for generic $g$. Moreover, we verify the contact structure on hyper-surfaces of the vortex dipole, and exclude the existence of perverse symmetric orbits. 

\tableofcontents 
\newpage

\section{Introduction}
\label{Chapter_Introduction}

\paragraph{}The birth of the $n$-vortex problem is marked by the paper \cite{helmholtz1867uber} of Helmholtz in 1858, and has been studied by many mathematicians and physicists since then, for example Kirchhoff \cite{kirchhoff1876} , Poincar\'e \cite{poincare1893theorie}, Arnold \cite{arnold1966geometrie} and so on. The system reveals deep insights on turbulence \cite{chorin2013vorticity} through the co-existence of regular and chaotic behaviours \cite{khanin1982quasi}, and provides a ``classical playground" \cite{aref2007point} for various branches of mathematics as well.

\paragraph{}The vortex system on a closed surface, initially known as a model in geophysics \cite{bogomolov1977dynamics}, has applications in various domains in today's real world: as large as the red spot on Jupiter, as tiny as the thin layers of liquid helium \cite{turner2010vortices}. Meanwhile, from the mathematical perspective, such system displays interesting interactions between dynamics and geometry. The geometric formulation of hydrodynamics traces back to Arnold \cite{arnold1966geometrie}, see also \cite{ebin1970groups} for details. The point $n$-vortex motion as dynamics on a special coadjoint orbit is studied in \cite{MarsdenCoadjoint}, see also the recent review of Boatto and Koiller \cite{boatto2015vortices} and an interpretation of Gustafsson \cite{gustafsson2019vortex} based on connections. The comprehensive book \cite{arnold1999topological} serves as a standard reference to this domain. Concerning the dynamical behaviours, since the work of Kimura \cite{kidambi1998motion} referring to the vortex dipole as geodesics detector, many specific symmetric surfaces (either discrete symmetry or continuous symmetry) have been investigated: the cylinder \cite{montaldi2003vortex}, the hyperbolic plane \cite{montaldi2014point} \cite{hwang2009point}, the toroidal surface \cite{sakajo2016point} and recently the triaxial ellipsoid \cite{rodrigues2018vortex} \cite{koiller2019vortex}, to name but a few. Attention has been given to the stability analysis of vortex rings \cite{dritschel2015motion} \cite{boatto2008curvature} \cite{boatto2019vortex}, reduction of system \cite{borisov2018three}, and orbits of the reduced integrable system \cite{souliere2002periodic}. However orbits on general surfaces without any extra symmetry, to the best knowledge of the author, are considerable less explored. 

\paragraph{}This article tends to make some initial attempts towards the closed invariant orbits of the $n$-vortex motion on a Riemann sphere $\mathbb{S}^2$ equipped with an arbitrary metric $g$ possibly without any non-trivial isometry group. As we will see soon, the $n$-vortex Hamiltonian system on the closed surface depends on both the vorticity vector $\mathbf{\Gamma}$ and the Riemannian metric $g$, via the symplectic form and the Hamiltonian function respectively. In the sequel, we will explore whether a particular choice of metric $g$ (hence the Riemannian structure) or that of vorticity vector $\mathbf{\Gamma}$ (hence the symplectic structure) will be necessary or irrelevant to produce, for instance, abundant fixed points and periodic orbits. 

\subsection{Background and Notations}
\label{Chapter_Introduction_Subsection_Background_and_Notations}

\paragraph{Notation:} We will henceforth use the following notations in the rest of the article:
\begin{itemize}
\item{\makebox[1.5cm]{$\mathbb{S}^2_{g}$}}: the Riemann sphere equipped with a Riemannian metric $g$;
\item{\makebox[1.5cm]{$(\mathbb{S}^2_{g})^n$}}: the product manifold of $n$ copies of $\mathbb{S}^2_g$, i.e., 
\begin{align*}
    (\mathbb{S}^2_{g})^n = \mathbb{S}^2_{g} \times \mathbb{S}^2_{g} \times ...\times \mathbb{S}^2_{g}
\end{align*}
\item{\makebox[1.5cm]{$\omega_g$}}: the Riemannian volume form;
\item{\makebox[1.5cm]{ $V_g(\mathbb{S}^2)$}}: the volume of $\mathbb{S}^2$ under the volume form $\omega_g$, defined as 
\begin{align*}
V_g(\mathbb{S}^2)=\int_{\mathbb{S}^2}\omega_g;
\end{align*}
\item{\makebox[1.5cm]{$\Delta$}}: the Laplace-de Rham operator, defined as $\Delta=\delta d + d\delta$;
\item{\makebox[1.5cm]{$\Delta_g$}}: the Laplace-Beltrami operator\footnotemark[1], defined as $\Delta_g=|g|^{-\frac{1}{2}}\partial_i (g^{ij}\sqrt{|g|}\partial_j)$;
\item{\makebox[1.5cm]{$H_g$}}: the Hamiltonian of the vortex problem on $\mathbb{S}^2_g$;
\item{\makebox[1.5cm]{$\mathcal{N}^n$}}: the collision set, defined as :
\begin{align*}
\mathcal{N}^n= \{ \mathbf{z}\in (\mathbb{S}^2_g)^n, \exists 1\leq i<j \leq n, z_i=z_j  \};
\end{align*}
\item{\makebox[1.5cm]{$\mathcal{N}^n_{\epsilon}$}}: the $\epsilon$-collision set, defined as ($\norm{\cdot}$ is understood as the Euclidean norm\footnotemark[2] in $\mathbb{R}^3$) :
\begin{align*}
\mathcal{N}^n_{\epsilon} = \{ \mathbf{z}\in (\mathbb{S}^2_g)^n, \exists 1\leq i<j \leq n, \norm{z_i-z_j} < \epsilon \};
\end{align*}
\item{\makebox[1.5cm]{ $\Sd$}}: the collision free configurations, defined as $(\mathbb{S}^2_g)^n \setminus \mathcal{N}^n$
\item{\makebox[1.5cm]{ $\Se$}}: the $\epsilon$-collision free configurations, defined as $(\mathbb{S}^2_g)^n \setminus \overline{\mathcal{N}^n_{\epsilon}}$;
\item{\makebox[1.5cm]{$\bar{\Gamma}$}}: the average vorticity, defined as $\bar{\Gamma}=\dfrac{1}{n}\sum_{i=1}^{n}\Gamma_i$.
\item{\makebox[1.5cm]{$U_{c,\epsilon}(H)$}}: A neighbourhood of a regular energy surface $S_c = H^{-1}(c)$, defined as 
\begin{align*}
    U_{c,\epsilon}(H) = \bigcup_{-\epsilon<\delta<\epsilon} H^{-1}(c+\delta) 
\end{align*} 
\end{itemize}

\footnotetext[1]{We have defined the Laplace-Beltrami \textbf{without} the minus sign, thus when acting on scalar functions the Laplace-de Rham operator $\Delta$ and Laplace-Beltrami operator differ by a minus sign: $\Delta$ = - $\Delta_g$. }
\footnotetext[2]{
We will denote by $d_g$ the distance induced on $\mathbb{S}^2_g$ by the Riemannian metric $g$. Note that for $\mathbb{S}^2$, any two metrics $g_1$ and $g_2$ will induce equivalent distance functions $d_{g_1}$ and $d_{g_2}$. $\norm{\cdot}$ stands for the Euclidean distance, when $\mathbb{S}^2_g$ is considered as an embedded submanifold in $\mathbb{R}^3$.}

\subsection{$N$-Vortex System:}
\label{Chapter_Introduction_Subsection_Vortex_System}
The motion of ideal fluid on $\mathbb{S}^2_g$ is governed by the Euler Equation in its velocity formulation 
\begin{align}
\label{EulerEquation}
\begin{cases}
\partial_{t} \mathbf{u} + \nabla_{\mathbf{u}}\mathbf{u}  = \nabla_{g} p \\
 div_g \mathbf{u} = 0 
 \end{cases}
 \tag{Velocity}
\end{align}
or in its vorticity formulation
\begin{align}
 \tag{Vorticity}
\frac{d\omega}{dt} + \mathcal{L}_{\mathbf{u}} \omega = 0
\end{align}
Here $\mathbf{u}\in \mathcal{T}\mathbb{S}^2_g$ is the velocity field of the fluid, $\omega = curl \mathbf{u}$ is the vorticity, $\nabla_{\mathbf{u}}$ is the covariant derivative along $\mathbf{u}$ 
and $\mathcal{L}_{\mathbf{u}}$ is the Lie derivative along $\mathbf{u}$.
The point vortices are then introduced by assuming that the vorticities are concentrated on finitely many Dirac measures, i.e. 
\begin{align}
\omega(t)= \sum_{i=1}^{n}\Gamma_i \delta_{z_i(t)}
\end{align}
Here $n$ is the total number of point vortices, $\Gamma_i\in \mathbb{R}_{*}$ is the vorticity of $i^{th}$ vortex, and $z_i(t)\in \mathbb{S}^2_g$ is the position of the $i^{th}$ vortex at time $t$. We will see in the next paragraph that the motion of vortices is governed by a Hamiltonian system, where the Hamiltonian function has singularity on the set $\mathcal{N}^n$, i.e., at the collision configurations. As a consequence the system is only defined on the collision-free set $\Sd$.

\paragraph{} Now consider the $n$-vortex problem on $\mathbb{S}^2_g$ with a fixed vorticity vector $\mathbf{\Gamma} = (\Gamma_1,\Gamma_2,...,\Gamma_n)$. Its phase space is then the symplectic manifold $((\mathbb{S}^2)^n_0, \Omega_g(\mathbf{\Gamma}))$, where the symplectic form $\Omega_g(\mathbf{\Gamma})$ reads:
\begin{align}
\label{Formula_symplectic_form}
\Omega_{g}(\mathbf{\Gamma}) = \Gamma_1\omega_g \oplus \Gamma_2 \omega_g \oplus...\oplus  \Gamma_n\omega_g 
\end{align}
The motion of the $n$ vortices on $\mathbb{S}^2_g$ is governed by the Hamiltonian system $\dot{\mathbf{z}}(t) = X_{H_{g}}(\mathbf{z}(t))$, with the Hamiltonian vector field $X_{H_{g}}$ defined as usual through
\begin{align}
\label{Formula_Hamiltonian_Vector_Field}
i_{X_{H_{g}}}\Omega_g (\mathbf{\Gamma})= dH_g
\end{align}
and the Hamiltonian equation is given by (see \cite{boatto2015vortices})
\begin{align}
\label{Formula_vortex_Hamiltonian}
H_g(\mathbf{z}) &= \sum_{1\leq i<j\leq n}\Gamma_i\Gamma_j G_g(z_i,z_j) + \sum_{1\leq i \leq n}\Gamma_i^2R_g(z_i)
\end{align}
In the above:
\begin{itemize}
\item $\mathbf{z}=(z_1,z_2,...z_n) \in \Sd$ denotes a collision free configuration of the $n$ individual vortices on the surface $\mathbb{S}^2_g$;
\item $G_g$ is the \textit{Green function} that solves the following equation in the sense of distribution :
\begin{align}
-\Delta_g G_g(z,w) = \delta_{z}(w) -\frac{1}{V_g(\mathbb{S}^2)} 
\end{align}
We can normalise Green function so that 
\begin{align}
\label{Formule_Green_Normalisation}
    \int_{\mathbb{S}^2}G(z,w) \omega_g(z) = 0  
\end{align}
\item $R_g$ is the \textit{Robin mass function}, which is the regular part of the Green function. It is defined by 
\begin{align}
R(z) = \lim_{w\rightarrow z} \big ( G_g(w,z) + \frac{1}{2\pi}\log d_g(w,z) \big )
\end{align}
where $d_g(w,z)$ denotes the distance between $w$ and $z$ on $\mathbb{S}^2_g$.  
\end{itemize}
Although in some cases one can try to find analytical formula of Green function (see \cite{koiller2019vortex,dritschel2015motion,yuuki2018green}), in general it is quite technical, making the explicit description of the dynamical equation rarely available.

\subsection{Isometry Group, Integrability, and Invariant Orbits}
\label{Chapter_Introduction_Subsection_Symmetry_of_Metrics}
The search of invariant orbits might be simplified if the metric has some non-trivial isometry group. Based on such isometry group a reduction can often be carried out. This will decrease the degree of freedom and simplify the search of invariant orbits. In this subsection we illustrate this fact by several examples:
\begin{Exp}[$\mathbf{SO}(3)$ symmetry]
\label{example_SO3}
Let $\mathbb{S}^2_{g_0}$ be the unit sphere in figure $1 (a)$, with $g_0$ the round metric\footnotemark. In the spherical coordinate $(\theta, \phi)$, a point $\eta =(\eta_x,\eta_y,\eta_z)\in \mathbb{S}^2 \subset \mathbb{R}^3$ is represented as
\begin{align*}
\eta_x = \sin\theta \cos \phi,  \quad \eta_y = \sin\theta \sin \phi,  \quad \eta_z = \cos \theta
\end{align*}
By choosing the symplectic coordinate $\omega = dp\wedge dq$ with 
\begin{align*}
p_i  = \cos \theta_i, \quad q_i  = \phi_i 
\end{align*}
and let $l_{ij}$ be the \textbf{chord length} between $z_i=(p_i,q_i)$ and $z_j=(p_j,q_j)$, The Hamiltonian is given by
\begin{align}
\label{Formula_Vortex_Round_Sphere}
H(\mathbf{z}) &= -\frac{1}{4\pi}\sum_{1\leq i<j\leq n}\Gamma_i\Gamma_j \log l_{ij}^2 \\
&= -\frac{1}{4\pi}\sum_{1\leq i<j\leq n}\Gamma_i\Gamma_j \log(1- \sqrt{1-p_i^2}\sqrt{1-p_j^2}\cos(q_i-q_j)-p_i p_j).\notag
\end{align}
Note that the system is invariant under the diagonal action of $\mathbf{SO}(3)$. The 3-vortex problem is thus integrable \cite{bogomolov1977dynamics,kidambi1998motion,sakajo1999motion}, while when  $n\geq 4$, the system is no longer integrable and there exists both chaos and KAM tori \cite{bagrets1997nonintegrability,lim1990kam}. For a detailed review of the $n$-vortex problem on the standard sphere, one can turn to the book of Newton \cite{newton2013n} and the references therein. 
\end{Exp}
\footnotetext{The round metric $g_0$ is induced by the natural Riemannian metric tensor on the Euclidean space $\mathbb{R}^3$ whose isometry group is $\mathbf{SO}(3)$. Taking $(\theta, \phi)$ as the standard spherical coordinate the metric is represented as $g_0= \begin{bmatrix} 1& 0\\
0& sin^2\theta\end{bmatrix}$}
\begin{figure}[ht]
\begin{center}
\includegraphics[width=140mm,scale=0.5]{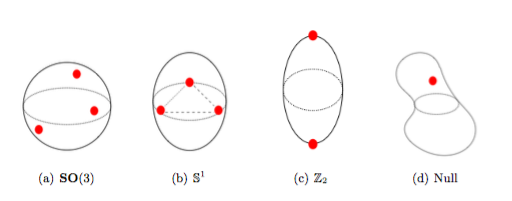}
\newline
\end{center}
\caption{Spheres with/without symmetry}
\end{figure}
\begin{Exp}[$\mathbb{S}^1$ symmetry]
Let $E_1$ be an ellipsoid of revolution in figure $1 (b)$, with $g_1$ its natural metric. By the uniformization theorem there exists a biholomorphic map $f: \mathbb{S}^2 \rightarrow E_1$ equivariant under the $\mathbb{S}^1$ symmetry. Then $\mathbb{S}_g^2$ with the pull-back metric $g = f^{*}g_{1}$ inherits the $\mathbb{S}^1$ symmetry with respect to the axis of rotation. As a result, if one puts identical $n$ vortices on the vertices of a $n$-polygon on a section perpendicular to the axis of rotation, they will rotate uniformly on a fixed latitude around the axis of rotation and become a relative equilibrium\footnotemark{}. 
\end{Exp}
\footnotetext{This means all vortices rotating uniformly around its centre of vorticity, thus becomes a fixed point in certain rotating frame.}
\begin{Exp}[$\mathbb{Z}_2$ symmetry \cite{koiller2019vortex}]
Let $E_2$ be an triaxial ellipsoid in figure $1 (c)$, with $g_2$ its natural metric. Such a manifold could be parametrised by $\frac{x^2}{a^2} + \frac{y^2}{b^2} + \frac{z^2}{c^2}=1$ with $0<a<b<c$. By the uniformization theorem there exists a biholomorphic map $f: \mathbb{S}^2 \rightarrow E_2$ equivariant under the symmetry with respect to the origin. Then $\mathbb{S}_g^2$ with the pull-back metric $g = f^{*}g_{2}$ inherits the $\mathbb{Z}_2$ symmetry with respect to the origin. As a result the three pairs of axis endpoints are natural candidates for fixed points of identical vortex dipole problem.
\end{Exp}
\begin{Exp}[No non-trivial symmetry]
Let $E_3$ be a general smooth topological sphere in figure $1 (d)$, with $g_3$ its natural metric. Such a manifold has no extra symmetric group. By the uniformization theorem there exists a biholomorphic map $f: \mathbb{S}^2 \rightarrow E_3$ . Then $\mathbb{S}_g^2$ with the pull-back metric $g = f^{*}g_{3}$ enjoys no extra symmetry neither. The only conserved quantity is the energy, since it is an autonomous Hamiltonian system. Even a $2$-vortex problem on $\mathbb{S}^2_g$ is plausible to be non-integrable due to the lack of symmetry, thus making the search of invariant orbits difficult for $n\geq 2$.
\end{Exp}

\subsection{Main Results} 
\label{Chapter_Introduction_Subsection_Main_Results}
The main results in this article concern the fixed points and the periodic orbits of the $n$-vortex problem on $\mathbb{S}^2_g$ for arbitrary $g$. To this end, we need to put some constraints on the vorticities. 
\paragraph{Non-degeneracy}
Let $\Lambda\subset \{1,2,3,...,n\}$ be a subset of indices, and denote 
\begin{align*}
\Gamma(\Lambda)= \sum_{i,j \in \Lambda, i\neq j } \Gamma_i \Gamma_j
\end{align*}

\begin{Def}[Non-degenerate vorticity]
A vorticity vector $\mathbf{\Gamma}$ is \textit{non-degenerate} if 
\begin{align}
\label{Condition_No_Collision}
\forall \Lambda \subset\{1,2,...,n\},  \quad \Gamma(\Lambda) \neq 0 \tag{P1}
\end{align} 
\end{Def}
As we know the main technical part in the analysis of the vortex is the singularity at collision. This constraint will prevent stationary configurations from accumulating into the collision set $\mathcal{N}^{n}$. As a result we can work on a compact subset of the configuration manifold. 
\paragraph{Thin Vorticity}
Next recall that given $n$ positive real numbers $\{\alpha_i\}_{1\leq i\leq n}$ , we call them \textit{commensurable} if $\alpha_i / \alpha_j \in \mathbb{Q}, \forall 1\leq i < j \leq n$. In particular, this means that 
\begin{align*}
  \exists \beta \in \mathbb{R}_{+}  \text{ and }  l_i \in \mathbb{N} \text{ s.t. } \alpha_i = \beta l_i \quad \forall 1\leq i\leq n
\end{align*}
Now let 
\begin{align*}
\kappa(\alpha_1,\alpha_2,...,\alpha_n) \coloneqq \beta h(l_1,l_2,...,l_n)
\end{align*}
where $h(l_1,l_2,...,l_n)$ denotes the highest common factor of $\{l_i\}_{1\leq i \leq n}$. In particular if we study the n-vortex problem on $\mathbb{S}^2_g$ with the vorticity vector $\mathbf{\Gamma} = (\Gamma_1, \Gamma_2,...,\Gamma_n)$, we can formulate the following definition:
\begin{Def}[Thin vorticity]
\label{Definition_Thin_Vortex}
We will say the $k^{th}$ vorticity $\Gamma_k$ is \textit{thin} with respect to $\mathbf{\Gamma}$ if 
\begin{align}
\label{Condition_Thin_Vorticity}
\Gamma_k \leq \kappa(\Gamma_1, \Gamma_2,..., \Gamma_{k-1}, \Gamma_{k+1},...,\Gamma_k) \tag{P2}
\end{align}
and we will call the corresponding vortex $z_k$ a \textit{thin vortex}.
\end{Def}
We describe the vorticity as being ``thin" because such property is related to the non-squeezing phenomena of symplectic embedding. Note that being thin is not an open property for the vorticity vector. 
\subsubsection{Theorems on Fixed Points}
Recall that $\mathbb{S}^2$ does not determine \textit{\`a priori} any Riemannian metric, but rather a unique conformally equivalent class of Riemannian metrics $[g]$, according to the uniformization theorem. Given any two Riemannian metrics $ g_1, g_2 \in [g]$ on $\mathbb{S}^2$ there exists a smooth function $\rho: \mathbb{S}^2 \rightarrow \mathbb{R}$ s.t. $g_2 = e^{2\rho}g_1$, and passing from $g_1$ to $g_2$ is called a conformal change of metric with conformal factor $\rho$. In particular, we use $g_0$ to represent the round metric defined in example \ref{example_SO3}. When there exists no risk of ambiguity, we will sometimes use $g(\rho)$ to denote the metric $g= e^{2\rho}g_0$ for short.
\paragraph{}Let $F$ be the critical set of the $n$-vortex problem with vorticity vector $\mathbf{\Gamma} = (\Gamma_1, \Gamma_2,...,\Gamma_n)$ on $\mathbb{S}^2_g$, i.e.
\begin{align*}
F= \{\mathbf{z}\in (\mathbb{S}^2_g)^n|dH_g(\mathbf{z})=\mathbf{0}\}
\end{align*}
Using a version of infinite dimensional transversality theorem, we prove that:
\begin{MainThm}
\label{Theorem_Main_Morse}
Fix a non-degenerate vorticity vector $\mathbf{\Gamma}$. For an open dense subset $\mathcal{D} \subset \mathcal{C}^{\infty}(\mathbb{S}^2, \mathbb{R})$,  consider $\mathbb{S}^2_g$ with $g= e^{2\rho} g_0, \rho \in \mathcal{D}$, then
\begin{enumerate}
    \item The Hamiltonian function $H_g$ is a Morse function on $\Sd$;
    \item There are only finitely many fixed points of the $n$-vortex problem on $\mathbb{S}^2_g$, i.e. $|F|<\infty$.
\end{enumerate}
\end{MainThm}
This theorem permits us to apply Morse theoretical argument for stationary $n$-vortex configuration on $\mathbb{S}^2_g$. It serves as the starting point for searching periodic orbits via either perturbative methods or variational methods. Moreover such non-degenerate fixed points could be used for prospective construction of steady vortex patches through implicit function theory, as in the planar case \cite{long2019concentrated}. 

We will henceforth refer to the open dense subset $\mathcal{D}$ in theorem \ref{Theorem_Main_Morse} as $\mathcal{D}_{M}$.

\subsubsection{Theorems on Periodic Orbits}
For the investigation of periodic orbits, we restrict ourselves to the positive vorticity, i.e. $\Gamma_{i}>0, \forall 1\leq i\leq n$. 
\paragraph{Existence of Periodic Orbits} First suppose that a Riemannian metric $g=g(\rho)$ s.t. $\rho \in \mathcal{D}_M$ is equipped on $\mathbb{S}^2$. Then with the presence of a thin vortex one can prove the existence of infinitely many periodic orbits for the $n$-vortex problem on $\mathbb{S}^2_g$:
\begin{MainThm}
\label{Theorem_Main_Periodic_Orbits}
Consider the $n$-vortex problem on $\mathbb{S}^2_g$ equipped with a Riemannian metric tensor $g$. Suppose that $\{\Gamma_i\}_{1\leq i\leq n} $ are all positive and that $\mathbf{\Gamma}$ possesses a thin vorticity, say $\Gamma_k$. Define for this index $k$ the following two values:
\begin{align}
c_1(k, g) = \min_{\eta\in \mathbb{S}^2_g} \min_{ \substack{\mathbf{z} \in \Sd \\ z_k = \eta}    }  H_g(\mathbf{z}),\quad c_2(k, g) = \max_{\eta\in \mathbb{S}^2_g} \min_{ \substack{\mathbf{z} \in \Sd \\ z_k = \eta}    }  H_g(\mathbf{z})
\end{align}
Then for $g=g(\rho)$ with $\rho \in \mathcal{D}_M$,
\begin{enumerate}
\item $-\infty < c_1(k,g) < c_2(k,g) < \infty$ strictly;
\item Let $I_g=(c_1(k,g), c_2(k,g))$, then $\forall c\in I_g$, there exists a periodic orbit $\{\mathbf{z}_j\}_{j\in \mathbb{N}}$ with 
\begin{align}
\lim_{j\rightarrow \infty} H_g(\mathbf{z}_j) = c
\end{align}
\end{enumerate}
\end{MainThm}
We prove this theorem using a result in the theory of J-holomorphic spheres. Hence the orbits obtained are not by method of bifurcation but of a variational nature instead. In particular, we see in theorem \ref{Theorem_Main_Periodic_Orbits} that one can specify the exact energy interval $I_g$ in which the method applies. 
\paragraph{Contact Structure of Vortex Dipole}
In particular if we consider the identical $2$-vortex problem, then each of them will be a thin vortex. We go one step further to show that for the vortex dipole problem on $\mathbb{S}^2_{g}$, there exists hyper-surfaces of contact type:  
\begin{MainThm}
\label{Theorem_Main_Contact}
Consider the identical $2$-vortex problem on $\mathbb{S}^2_g$. Then for $\norm{\rho}_{\mathcal{C}^1}$ small there exists a constant $c_0$ s.t. for any $c>c_0$ the hyper-surface $H^{-1}_{g}(c)$ is of contact type.
\end{MainThm}
Theorem \ref{Theorem_Main_Contact} could be combined with theorem \ref{Theorem_Main_Periodic_Orbits} to look for periodic orbit exactly on a given hyper-surface.
\paragraph{Exclusion of Perverse Orbits} Finally we study the existence of perverse orbit. In the n-body problem and the n-vortex problem many symmetric periodic orbits are found, in particular the choreography :
\begin{Def}
A choreography of the $n$-vortex problem is a $T$-periodic orbit of the system where the $n$ individual vortices equally spread (in time) along a single closed curve\footnotemark, i.e.,
\begin{align}
z_{i-1}(t)=z_i(t+\frac{T}{n}), \quad \forall 1\leq i \leq n, t\in \mathbb{R}.
\end{align}
\end{Def}
\footnotetext{As a common convention $z_{n+1} = z_{0}$.}
Although we don't know whether a choreography always really exists, we see that apart from the identical vorticity case, the vorticity vector do not induce any non-trivial gauge group for all vortices. Hence it is natural to ask whether there exists a perverse choreography, i.e., choreography with non-identical vorticities. By adapting ourselves to the setting in \cite{ChencinerPerverse}, we prove that:  
\begin{MainThm}
\label{Theorem_Main_No_Perverse_Choregraphy}
Suppose that $\{\Gamma_i\}_{1\leq i\leq n}$ are not all identical. Then for $\mathbb{S}^2_g$ with $g= g(\rho), \rho \in \mathcal{D}_M$ the system possesses no choreography.
\end{MainThm} 
This theorem is consistent with the results given in \cite{Celli2003On} for the $n$-vortex in the plane. 

\paragraph{} The methods presented in this article might contribute in several aspects to the subject of the $n$-vortex problem on surfaces :
\begin{itemize}
\item{\textit{generic metrics}:} compared to fixed points and periodic orbits of the $n$-vortex problem on symmetric closed surfaces before, our approach does not rely on the symmetry of these metrics and works for an open dense subset of Riemannian metric tensor on $\mathbb{S}^2$.
\item{\textit{arbitrary number and flexible vorticity }:} our approach could apply for arbitrary number of point vortices even in the non-integrable cases, such property is important if one wants to approximate the smooth vorticity field by increasing the number of vortices. Moreover the method applies to more general choices of vorticities, including the identical vorticity case. 
\end{itemize}


\section{Conformal Metric, Transversality and Fixed Points}
\label{Chapter_Riemannian_Metric_and_Hamiltonian_Function}
The application of Morse theory to the relative equilibra of the planar $n$-vortex problem is noticed by Palmore \cite{palmore1982relative}. A detailed investigation is provided in the recent work of Roberts \cite{roberts2018morse}. Such ideas trace back to Smale \cite{Smale2006Problems} who suggests in 2006 the Morse theory as a promising approach towards the understanding of the $5^{th}$ of his 21-century problems \cite{smale1998mathematical}, i.e. the finiteness of central configurations. In the case of the $n$-vortex problem on $\mathbb{S}^2_g$ with an arbitrary metric $g$, We will study absolute equilibra directly since relative equilibrium might not exist unless $g$ is $\mathbb{S}^1$ symmetric. 

Note that normally the vorticity vector $\mathbf{\Gamma} = (\Gamma_1, \Gamma_2,...,\Gamma_n)$ is taken as the parameter set. Our approach, however, keeps the vorticity vector invariant but relies on the variation of metric tensor, thus coming with an infinite dimensional parameter set instead. Such an approach is motivated by the recent work of Bartsch et al. \cite{bartsch2017morse} on the non-degeneracy of Kirchoff-Robin function in a planar bounded domain.

\subsection{Conformal Change of Metric}
\label{Chapter_Riemannian_Subsection_Conformal_Change_of_Metric}
To investigate the non-degeneracy of fixed point under a perturbation of Riemannian metric, it is natural to ask first how the Hamiltonian $H_g$ changes as $g$ varies in the unique conformal class $[g]$. This is closely related to the study of geometric mass, see \cite{steiner2005geometrical}. One has that: 
\begin{Pro}[Theorem 4, \cite{boatto2015vortices}]
Consider a conformal change of metric $g_1 \rightarrow g_2: g_2 = e^{2\rho} g_1$. The two Hamiltonians before and after the change of metric are related by:
\begin{align}
\label{Formula_Hamiltonian_under_a_Conformal_Change_of_Metric}
H_{g_2}(\mathbf{z}) = H_{g_1}(\mathbf{z}) +\frac{1}{2\pi}\sum_{i=1}^{n}\Gamma_i^2 \rho(z_i) - \frac{\sum_{i=1}^{n}\Gamma_i}{V_{g_2}(\mathbb{S}^2)}\sum_{i=1}^{n}\Gamma_i \Delta^{-1}_{g_1}e^{2\rho(z_i)}
\end{align}
\end{Pro}

\begin{Rmk}

\label{remark:weakened regularity of conformal factor}
If we only assume that $g_1,g_2$ are tensors of class $\mathcal{C}^{m}$, instead of being smooth ones, then we can enlarge the space of conformal factors from 
$\mathcal{C}^{\infty}$ to $\mathcal{C}^{m}$ and consider all metrics of the type $g=e^{2\rho} g_0$ for $\rho \in \mathcal{C}^{m}(\mathbb{S}^2,\mathbb{R})$. Formula \eqref{Formula_Hamiltonian_under_a_Conformal_Change_of_Metric} together with its gradient and Hessian (at its critical points) still makes sense for $m\geq 2$. We will abuse the terminology by calling it a perturbation of metric even if $\rho\in \mathcal{C}^{m}(\mathbb{S}^2, \mathbb{R})$. 
\end{Rmk}
In particular, if $\rho(z)= \rho \in \mathbb{R}$ is a constant, we call $g\rightarrow e^{2\rho}g$ a \textit{homothetic} change of metric. Dynamically, orbit of the $n$-vortex problem on $\mathbb{S}^2_g$ stays invariant under such a change of metric, as the following lemma shows:
\begin{Col}
\label{Corollary_homothetic_change_of_metric}
Suppose that $g_1= e^{2\rho} g_2$ is a homothetic change of metric for some constant $\rho\in \mathbb{R}$. Then $H_{g_1}$ differs from $H_{g_2}$ only by a constant .
\end{Col}
\begin{proof}
By formula \eqref{Formula_Hamiltonian_under_a_Conformal_Change_of_Metric} and \eqref{Formule_Green_Normalisation} one sees that 
\begin{align}
    H_{g_2}(\mathbf{z}) = H_{g_1}(\mathbf{z}) -\frac{1}{2\pi}\sum_{i=1}^{n}\Gamma_i^2 \rho
\end{align}
  
\end{proof}
This lemma implies that as long as only dynamical aspect is concerned, one can always assume $V_g(\mathbb{S}^2)= \alpha$ for some $\alpha \in \mathbb{R}_{+}$.  
\subsection{Transversality and the Morse Property}
\label{Chapter_Riemannian_Subsection_Transversality}
\paragraph{}Our aim is to show that for most of Riemannian metric tensors $g\in [g]$, the mapping $d H_g(\mathbf{z})$ will have $\mathbf{0}$ as a regular value. We will need the following general transversality theorem:
\begin{theorem}
\label{Theorem_parametrized_version_of_Sard_Smale}
\cite[Theorem 5.10]{henry2005perturbation}
Given $k \in \mathbb{N}$, consider the $\mathcal{C}^{k}$ Banach manifolds $X,Y,Z,W$ with $W\subset Z$ a $\mathcal{C}^{k}$ submanifold of $Z$, and a $\mathcal{C}^{k}$ map 
\begin{align*}
f: X\times Y \rightarrow Z
\end{align*}
Assume that 
\begin{description}
\item{(H1)} $X$ is finite dimensional and $\sigma$-compact, $W$ is $\sigma$-closed, and
\begin{align}
dim(X) - codim(W) < k
\end{align}
\item{(H2)} For each $(x,y) \in f^{-1}(W)$, with $z = f(x,y)$, one has that 
\begin{align}
Range(Df(x,y)) + \mathcal{T}_z W = \mathcal{T}_z Z
\end{align}
\item{(H3)} The map 
\begin{align*}
f^{-1}(W) &\rightarrow W \times Y\\
(x,y) &\rightarrow (f(x,y),y)
\end{align*}
is $\sigma$-proper.
\end{description}
Then  the set 
\begin{align}
Y_{crit} = \{y\in Y \text{}| \text{ } f(\cdot, y) \text{ is not transverse to } W  \} 
\end{align}
is meager in Y.
\end{theorem}
We would like to apply the above theorem to study the critical points of Hamiltonian $H_g$ of the $n$-vortex problem on $\mathbb{S}^2_g$. To this end let $k =2$ and assign
\begin{align*}
X = \Se, \quad Y = \mathcal{C}^{2}(\mathbb{S}^2_g, \mathbb{R}),\quad  Z = \mathcal{T}\Se, \quad W =s_0(\Se)
\end{align*}
where $s_0\in \Gamma(\mathcal{T}^*\Se)$ is the zero section, hence $W$ becomes a sub-manifold of $Z$. 
Moreover define the map $f$ to be
\begin{align}
\label{Formula_Gradient_Hamiltonian}
f:\Se \times \mathcal{C}^{2}(\mathbb{S}^2, \mathbb{R})
\rightarrow  \mathcal{T}^{*}\Se \notag \\
(\mathbf{z}, \rho) \xrightarrow{f} dH_{g(\rho)}(\mathbf{z}) \tag{F}
\end{align}
where $g(\rho) = e^{2\rho}g_0$ with $\rho \in \mathcal{C}^{2}(\mathbb{S}^2, \mathbb{R})$ and $g_0$ the round metric. If $(\mathbf{z},\rho) \in f^{-1}(W)$, then $\mathbf{z}$ is a critical point of the $n$-vortex Hamiltonian $H_{g(\rho)}$ on $\mathbb{S}^2_{g(\rho)}$. Recall that when restricted to the zero section one has that 
\begin{align}
\label{Formula_split_vector_bundle}
\mathcal{T}(\mathcal{T}^*\Se) = \mathcal{T}\Se \oplus \mathcal{T}^*\Se
\end{align}
To adapt ourselves to the assumptions in theorem \ref{Theorem_parametrized_version_of_Sard_Smale}, we first calculate its linearization.
\begin{Lem}
\label{Lemma_transversality}
The linearization $Df$ of the map $f$ :
\begin{align*}
Df_{(\mathbf{z},\rho)}: \mathcal{T}_{\mathbf{z}}\Se \times \mathcal{T}_{\rho}\mathcal{C}^{2}(\mathbb{S}^2, \mathbb{R}) \rightarrow \mathcal{T}_{({\mathbf{z},dH_{g(\rho)}}(\mathbf{z}))} (\mathcal{T}^{*}\Se)
\end{align*}
is continuous, moreover at  $(\mathbf{z}, \rho)\in f^{-1}(W)$ 
it reads:
\begin{align}
\label{Formula_linearised_operator}
Df_{(\mathbf{z},\rho)}(\mathbf{w}, \psi) &= \big(\mathbf{w}, d^2 H_{g(\rho)}(\z)(\mathbf{w})+\frac{1}{2\pi} \sum_{i=1}^n \Gamma_i^2  d\psi (z_i)  
-2d \big ( \frac{\sum_{i=1}^{n}\Gamma_i}{V_{g(\rho)}(\mathbb{S}^2)} \sum_{i=1}^n \Gamma_i \Delta_{g(\rho)}^{-1} \psi(z_i) \big )
 \big)
\end{align}
\end{Lem}
\begin{proof}
See appendix \ref{App_Chapter_Linearised_differential}.
\end{proof}
The following elementary lemma is used in the later proof for surjectivity of differentials.

\begin{Pro}
\label{Proposition_surjectivity}
Define for fixed $z\in \mathbb{S}^2_g$ and $\alpha \in \mathbb{R}, \alpha \neq 0$ the following linear map 
\begin{align}
\phi_{z,\alpha}: \mathcal{C}^{2}(\mathbb{S}^2_g, \mathbb{R}) &\rightarrow \mathcal{T}_{z}^{*}\mathbb{S}^2_g \simeq \mathbb{R}^2\notag\\
\phi_{z,\alpha}(\psi) &= d((\Delta_g^{-1}+\frac{1}{\alpha} I) \psi)|_z 
\end{align}
Let $\mathbf{z}=(z_1,z_2,...,z_n) \in \Se$ and $\mathbf{\beta}=(\beta_1,\beta_2,...,\beta_n) \in \mathbb{R}^n$ s.t. 
\begin{align}
\beta_i\notin \sigma(-\Delta_g), \forall 1\leq i\leq n
\end{align}
Then the map
\begin{align}
\Phi_{\mathbf{z},\mathbf{\beta}}:\mathcal{C}^{2}(\mathbb{S}^2_g, \mathbb{R}) &\rightarrow \mathcal{T}_{\mathbf{z}}^{*}(\mathbb{S}^2_g)^n \simeq \mathbb{R}^{2n} \notag \\
\Phi(\psi) &= (\phi_{z_1,\beta_1}(\psi),\phi_{z_2,\beta_2}(\psi),...,\phi_{z_n,\beta_n}(\psi) )
\end{align}
is surjective.
\end{Pro}
\begin{proof}
If $\beta_i\notin \sigma(-\Delta_g)$, then $\frac{1}{\beta_i} \notin \sigma(-\Delta_g^{-1})$. As a result elementary spectrum theory implies that the following operator $L_{\beta_i}$
\begin{align}
\mathcal{C}^{2}(\mathbb{S}^2_g, \mathbb{R}) \rightarrow \mathcal{C}^{2}(\mathbb{S}^2_g, \mathbb{R})\notag\\
\psi \xrightarrow{L_{\beta_i}} (\Delta_g^{-1}+\frac{1}{\beta_i} I) \psi
\end{align}
has a dense image in $\mathcal{C}^{2}(\mathbb{S}^2_g, \mathbb{R})$. Hence $\phi_{z,\beta_i}$ is onto. 

Now let $\eta = (x_1,y_1,x_2,y_2,...,x_n,y_n) \in \mathcal{T}_{\mathbf{z}}^{*}\Se$. There exists a function $\psi_i$ s.t. $\phi_{z_i,\beta_i} (\psi_i)= (x_i,y_i)$. Since $\mathbf{z}\in \Se$, by definition $z_i \neq z_j$ if $i\neq j$. One can then choose disjoint neighbourhoods $U_i$ of $z_i$ and a compact set $K_i$ s.t. $z_i \in K_i \subset U_i$ together with a smooth bump function $h_i$. s.t. 
\begin{align}
h_i(z)=
\begin{cases}
  1, \quad z\in K_i\\
  0, \quad z\in \mathbb{S}_g^2 \setminus U_i
\end{cases}
\end{align}
Then $\displaystyle \psi(z)  = \sum_{1\leq i \leq n} h_i(z) \psi_i(z)$ is a desired function s.t. $\Phi(\psi)=\eta$. The corollary is proved.    
\end{proof}
Now we go back to the study the surjectivity of the linearisation $Df$. Let 

\begin{align*}
\beta_i  = 
\frac{ - 4n\pi \bar{\Gamma}}{\Gamma_i V_{g}(\mathbb{S}^2)},\quad 1\leq i\leq n.
\end{align*}
 
\begin{Pro}
\label{Proposition_transversality}
Suppose that 
\begin{align}
\label{Condition_Vorticity_Spectrum}
\tag{P3}
\{\beta_i \}_{1\leq i\leq n} \bigcap \sigma(-\Delta_g) \in\{ \emptyset , \{0\} \}
\end{align}
Then the linearisation $Df$ is onto for each $(\mathbf{z},\rho) \in f^{-1}(\mathbf{0})$.
\end{Pro}
\begin{proof}
We only need to show the fiberwise surjectivity. To this end, for any $\mathbf{\eta}= (x_1,y_1,x_2,y_2,...,x_n,y_n)\in\mathcal{T}_{\mathbf{z}}^{*}(\Se \simeq \mathbb{R}^{2n} $, 
if $d^2(H_{g(\rho)}(\mathbf{z}))$ is non-singluar, one can choose $d\psi = 0$ and clearly $d^2(H_{g(\rho)}(\mathbf{z})): \mathbb{R}^{2n} \rightarrow \mathbb{R}^{2n}$ is onto, and we are done.

If $d^2(H_{g(\rho)}(\mathbf{z}))$ is indeed singular, then let $\tilde{\eta} = \eta - d^2(H_{g(\rho)}(\mathbf{z}))\mathbf{w}=(\tilde{x}_1,\tilde{x}_2,...,\tilde{x}_n,\tilde{y}_n)$ we will now encounter two possible situations:
\begin{enumerate}
\item If $\{\beta_i \}_{1\leq i\leq n} \bigcap \sigma(-\Delta_g) =\{0\}$, then $\bar{\Gamma} = 0$ and 
\begin{align}
Df_{(\mathbf{z},\rho)}(\mathbf{w}, \psi) = d^2(H_{g(\rho)}(\mathbf{z}))\mathbf{w} + \frac{1}{2\pi}\sum_{i=1}^{n}\Gamma_i^2 d\psi(z_i) 
\end{align}
Thus we should find a solution for the equation 
\begin{align}
    \frac{1}{2\pi}\sum_{i=1}^{n}\Gamma_i^2 d\psi(z_i) =\tilde{\eta}
\end{align}
This can be achieved by simply constructing $\psi(z) = \sum_{1\leq i \leq n} h_i(z) \psi_i(z)$, with $d\psi_i(z_i) = \frac{2\pi}{\Gamma_i^2}(\tilde{x}_i,\tilde{y}_i)$ and $\{h_i\}_{1\leq i\leq n}$ the bump functions in corollary \ref{Proposition_surjectivity};
\item If  $\{\beta_i \}_{1\leq i\leq n} \bigcap \sigma(-\Delta_g)= \emptyset$, then one has that :
\begin{align*}
Df_{(\mathbf{z},\rho)}(\mathbf{w}, \psi) 
&=d^2(H_{g(\rho)}(\mathbf{z}))\mathbf{w} -  \frac{2n\bar{\Gamma}}{V_{g(\rho)}(\mathbb{S}^2)}\sum_{i=1}^{n}\Gamma_i d(\Delta_{g(\rho)}^{-1} -\Gamma_i\frac{V_{g(\rho)}(\mathbb{S}^2)}{4n\pi \bar{\Gamma}}I )\circ \psi\\
&= d^2(H_{g(\rho)}(\mathbf{z}))\mathbf{w} -  \frac{2n\bar{\Gamma}}{V_{g(\rho)}(\mathbb{S}^2)}\sum_{i=1}^{n}\Gamma_i d((\Delta_{g(\rho)}^{-1} +\frac{1}{\beta_i} I )\psi)\\
&= d^2(H_{g(\rho)}(\mathbf{z}))\mathbf{w} -  \frac{2n\bar{\Gamma}}{V_{g(\rho)}(\mathbb{S}^2)} \mathbf{\Lambda} \Phi_{\mathbf{z},\mathbf{\beta}}(\psi)
\end{align*}
Thus we should find a solution for the equation 
\begin{align}
 \mathbf{\Lambda} \Phi_{\mathbf{z},\mathbf{\beta}}(\psi) =-  \frac{V_{g(\rho)}(\mathbb{S}^2)}{\bar{\Gamma}}\tilde{\eta}
\end{align}
where the diagonal matrix
\begin{align}
\mathbf{\Lambda}  = diag(\Gamma_1, \Gamma_1, \Gamma_2 , \Gamma_2,...,\Gamma_n,\Gamma_n)
\end{align}
The surjectivity then follows from proposition \ref{Proposition_surjectivity} and that $\mathbf{\Lambda}$ is invertible.   
\end{enumerate}
 
\end{proof}
We can now prove the non-degeneracy of the Hamiltonian up to a generic $\mathcal{C}^{2}$ perturbation:
\begin{theorem}
\label{Thm_local_open_dense}
Suppose that for some $g\in[g]$ the condition \eqref{Condition_Vorticity_Spectrum} holds. Then for any $\epsilon>0$, up to a generic perturbation of $\rho \in \mathcal{C}^{2}(\mathbb{S}^2,\mathbb{R})$ s.t. $\tilde{g} = e^{2\rho} g$, $H_{\tilde{g}}$ has only non-degenerate critical points in $\Se$.
\end{theorem}
\begin{proof}
One only needs to show that theorem \ref{Theorem_parametrized_version_of_Sard_Smale} applies. To this end, denote $\epsilon_k = \epsilon+\frac{1}{k}$, we verify that $(H1)-(H3)$ are true.
\begin{itemize}
\item{For $(H1)$:} since $\displaystyle \Se=\bigcup_{k\in \mathbb{N}}\overline{(\mathbb{S}^2)^n_{\epsilon_k}}$, i.e., $\Se$ is $\sigma$-compact. Moreover 
\begin{align*}
dim(X) - codim(W) = dim(\Se) - codim(s(\Se)) <2
\end{align*}
Hence $(H1)$ holds.
\item{For $(H2)$:} we see that $Df$ is subjective due to proposition \ref{Proposition_transversality}, in other words (H2) holds.
\item {For $(H3)$:} $\Se$ is $\sigma$-compact and that $W$ is $\sigma$-closed (as $\mathcal{T}^* \Se$ is metrizable). Hence $(H3)$ holds.
\end{itemize}
The theorem is thus proved. 
 
\end{proof}
Last but not least, observe that one can ``forget" the constraint \eqref{Condition_Vorticity_Spectrum}:
\begin{Lem}
\label{Lemma_Not_in_Specturm}
Fix a vorticity vector $\mathbf{\Gamma}$. Let 
\begin{align*}
    \mathcal{G}(\mathbf{\Gamma}) = \{\rho |  \text{  condition \eqref{Condition_Vorticity_Spectrum} holds for } g=e^{2\rho}g_0 \}
\end{align*}
Then $ \mathcal{G}(\mathbf{\Gamma})$ is open dense in $\mathcal{C}^{\infty}(\mathbb{S}^2,\mathbb{R})$.
\end{Lem}
\begin{proof}
It is well known that spectrum $\sigma(-\Delta_{g})$ is non negative, starting from a simple 0, discrete, of finite multiplicity and increases to $+\infty$.
If $\bar{\mathbf{\Gamma}}=0$, then there is nothing to prove as $0\in\sigma(-\Delta_{g})$ for any $g$. 
Now consider the case $\bar{\mathbf{\Gamma}}\neq 0$. Suppose that $g$ s.t. $\{\beta_i \}_{1\leq i\leq n} \bigcap \sigma(-\Delta_g)=\emptyset$. Then let $\beta_{max}=\max_{1\leq i\leq n}\{\beta_i\}$, the interval $(0,\beta_{max})$ has only finitely many eigenvalues in $\sigma(-\Delta_{g})$, whose continuous dependence on the metric is thus uniform. Hence $\mathcal{G}(\mathbf{\Gamma})$ is open. Next, suppose $\rho\notin \mathcal{G}(\mathbf{\Gamma})$ and consider the homothetic change of metric $\tilde{\rho}= \rho + \frac{1}{2}\log c$ for $c$ close to $1$, thus $\tilde{g} = c g $. Note that $\beta \in \sigma(-\Delta_{g})\Leftrightarrow c\beta\in \sigma(-\Delta_{\tilde{g}})$. But $\beta_{i}$ becomes $\tilde{\beta}_{i} = \frac{1}{c} \beta_i$. Hence up to an arbitrary small homothetic change of metric condition \eqref{Condition_Vorticity_Spectrum} will be true, which implies that $\mathcal{G}(\mathbf{\Gamma})$ is dense.  
\end{proof}
Note that so far our non-degeneracy of $H_{g(\rho)}$ for generic $\rho\in \mathcal{C}^{2}(\mathbb{S}^2,\mathbb{R})$ is achieved in $\Se$ for any fixed $\epsilon>0$. In particular this is not a compact manifold. Given a $g= g(\rho)$, it will be convenient if we know \textit{\`a priori} that all the fixed points are uniformly separated away from the collision set. This will be studied in the next sub-section by assuming the non-degenerate condition on vorticity vector $\mathbf{\Gamma}$. 

\subsection{Critical Points and the Collision Set}
\label{Chapter_Riemannian_Subsection_Critical_Points_Collision}
\paragraph{}In celestial mechanics, the Shub's lemma \cite{shub1971appendix} states that the relative equilibria are uniformly separated from the set of collisions. This is also true for vortex problem on the plane if \eqref{Condition_No_Collision} holds by using the fact that $<dH(\mathbf{z}),\mathbf{z}>=cst$, see \cite{o1987stationary} \cite{roberts2018morse}. Although on $\mathbb{S}^2_g$ we can no longer take profit of this relation, in this paragraph we show that similar analogue holds. The intuition here is quite natural: as some of the vortices approach to each other into a small cluster and accumulate to a single point, replacing them by a single vortex which bears the sum of vorticities in this cluster will lead to a stationary configuration for fewer vortices under consideration.
\begin{Lem}
\label{Lemma_Accumulation_Cluster}
Consider the n-vortex problem on the Riemann sphere equipped with the metric $g$. If there exist an index set $\Lambda\subset\{1,2,...,n\}$ and a sequence of stationary configurations $\{ \mathbf{z}^k = (z_{1}^k,z_2^{k},...,z_{n}^k) \in \mathbb{S}^2_g \}_{k \in \mathbb{N}}$ satisfying that 
\begin{align}
\exists z^{*}\in \mathbb{S}^2_g , \text{ s.t. } \lim_{k\rightarrow \infty}z^k_i= z^{*} \quad  \forall i\in \Lambda
\end{align}
Then one must have that 
\begin{align}
\sum_{i\neq j, i,j\in \Lambda}\Gamma_i \Gamma_j = 0
\end{align}

\end{Lem}
\begin{proof}
We first represent Hamiltonian $H_g$ by $H_{g_0}$, while $g_0$ is the round metric and $g= e^{2\rho}g_0$ is the conformal change of metric. Actually putting \eqref{Formula_Hamiltonian_under_a_Conformal_Change_of_Metric} and \eqref{Formula_Vortex_Round_Sphere}, one sees that 
\begin{align}
H_g(\mathbf{z}) &= H_{g_0}(\mathbf{z})  +\frac{1}{2\pi}\sum_{i=1}^{n}\Gamma_i^2 \rho(z_i)-\frac{\sum_{i=1}^{n}\Gamma_i}{V_{g}(\mathbb{S}^2)}\sum_{i=1}^{n}\Gamma_i \Delta^{-1}_{g_0}e^{2\rho(z_i)}\\
&= -\frac{1}{4\pi}\sum_{1\leq i<j\leq n}\Gamma_i\Gamma_j \log l_{ij}^2  +\frac{1}{2\pi}\sum_{i=1}^{n}\Gamma_i^2 \rho(z_i)-\frac{n\bar{\Gamma}\sum_{i=1}^{n}\Gamma_i \Delta^{-1}_{g_0}e^{2\rho(z_i) }}{V_{g}(\mathbb{S}^2)}
\end{align}
One can explicitly write down the dynamics of $n$-vortex seen as a system embedded in $\mathbb{R}^3$, i.e.
\begin{align}
\Gamma_i \dot{s}_{i}^k = -\frac{1}{e^{2\rho(s_i^k)}} \big(  \sum_{\substack{1\leq j\leq n\\ j\neq i}} \Gamma_i\Gamma_j \frac{s_j^k \times s_i^k}{\norm{s_i^k-s_j^k}^2} + \Gamma_i^2 F(s_i^k) \big)
\end{align}
here $s^{k}_i \in \mathbb{R}^3$ represent the coordinates of the $i^{th}$ vortex in $\mathbb{R}^3$ of the $k^{th}$ stationary configuration. The vector $F$ corresponds to the dynamics generated from Robin's mass function due to the change of metric. $F$ depends only on the surface geometry and is indifferent to the interaction between distinguished vortices. 

Now given an index set $\Lambda \in \{1,2,...,n\}$ and a convergent sequence of stationary configurations $\{ \mathbf{s}^{k}\}_{k\in \mathbb{N}}$ with $ \mathbf{s}^k =(s^k_1, s^k_2,...,s^k_n)\in \mathbb{R}^{3n}, \forall k\in \mathbb{N}$ satisfying that 
\begin{align*}
\forall 1\leq i\leq n, \lim_{k\rightarrow \infty}s^k_i = s^{*}_i \in \mathbb{S}^2_g\quad \text{ and }\quad \forall i\in \Lambda, s^*_i = s^*
\end{align*}
 For $i\in \Lambda$, we look at the interaction due to the vortices in the group $\Lambda$ and those not in the group $\Lambda$ respectively. More precisely,
\begin{align}
\mathbf{0} =\Gamma_i \dot{s}_{i}^k = -\frac{1}{e^{2\rho(s_i^k)}} (  \sum_{j\in \Lambda,j\neq i} \Gamma_i\Gamma_j \frac{s_j^k \times s_i^k}{\norm{s_i^k-s_j^k}^2} + \sum_{l\notin \Lambda} \Gamma_i\Gamma_l \frac{s_l^k \times s_i^k}{\norm{s_i^k-s_l^k}^2} + \Gamma_i^2 F(s_i^k))
\end{align}
Multiplying $\frac{1}{e^{2\rho(s_i^k)}}$ and summing over $i\in \Lambda$ implies that
\begin{align}
\mathbf{0}=\sum_{i\in \Lambda} e^{2\rho(s_i^k)}\Gamma_i \dot{s}_{i}^k = -\sum_{i\in \Lambda} \sum_{l\notin \Lambda} \Gamma_i\Gamma_l \frac{s_l^k \times s_i^k}{\norm{s_i^k-s_l^k}^2} -\sum_{i\in \Lambda} \Gamma_i^2 F(s_i^k) 
\end{align}
By passing $k\rightarrow \infty$ one sees that 
\begin{align}
\label{Formula_cancelled}
\sum_{i\in \Lambda} \sum_{l\notin \Lambda} \Gamma_i\Gamma_l \frac{s^*_l \times s^*}{\norm{s^*_l-s^*}^2} + \sum_{i\in \Lambda} \Gamma_i^2 F(s^*) = \mathbf{0}
\end{align}

Next notice that by calculating the cross product of $s_i^{k}\times \dot{s}_i^{k}$, one sees that  
\begin{align}
\mathbf{0}&=\sum_{i\in \Lambda}e^{2\rho(s_i^k)}\Gamma_i s_i^k \times  \dot{s}_{i}^k \notag\\
&= - \frac{1}{2}\sum_{i\neq j, i,j \in \Lambda} \Gamma_i\Gamma_j (s_i^k +s_j^k) 
- \sum_{i\in \Lambda} s_i^k \times (\sum_{l\notin \Lambda} \Gamma_i\Gamma_l \frac{s_l^k \times s_i^k}{\norm{s_i^k-s_l^k}^2} +\Gamma_i^2 F(s_i^k))
\end{align}
The fact that $\forall i\in \Lambda, \lim_{k\rightarrow \infty} s_i^k = s^* $ together with formula \eqref{Formula_cancelled} implies then
\begin{align}
\label{Eq:sum_sub_vorticity}
\sum_{\substack{i\neq j\\ i,j \in \Lambda}} \Gamma_i\Gamma_j s^* = \mathbf{0}  
\end{align}
Since $s^*\in \mathbb{S}^2$, \eqref{Eq:sum_sub_vorticity} can hold only when $\displaystyle\sum_{\substack{i\neq j\\ i,j \in \Lambda}} \Gamma_i\Gamma_j =0$.  
\end{proof}
We are now at the point to prove theorem \ref{Theorem_Main_Morse}:
\paragraph{\textbf{Proof of Theorem \ref{Theorem_Main_Morse}}:} 
We fix $\mathbf{\Gamma}= (\Gamma_1, \Gamma_2,...,\Gamma_n)$ s.t. the condition \eqref{Condition_No_Collision} holds. For generic $\rho \in \mathcal{C}^{2}(\mathbb{S}^2,\mathbb{R})$, the Hamiltonian $H_{g(\rho)}$ is non-degenerate due to theorem \ref{Thm_local_open_dense} and lemma \ref{Lemma_Not_in_Specturm}. Now according to lemma \ref{Lemma_Accumulation_Cluster} no fixed points can accumulate to the diagonal $\mathcal{N}^n$, in other words there exists an $\epsilon$ s.t. all the fixed points of $H_g$ are located in the closure of $\Se$, which is a compact set and $H_{g(\rho)}$ can has only finitely many fixed points in the compact set $\Se$. The density of $\mathcal{C}^{\infty}(\mathbb{S}^2,\mathbb{R})$ in $\mathcal{C}^{2}(\mathbb{S}^2,\mathbb{R})$ then finishes the proof.    \QED


\section{Vorticity, Symplectic Form, and Periodic Orbits}
\label{Chapter_Vorticity_and_Symplectic_Invariant}
In this part we study the $n$-vortex problem on $\mathbb{S}^2_g$ from the perspective of symplectic geometry. We will also need to use the existence of non-trivial pseudo-holomorphic sphere in $(\mathbb{S}^2_g)^n$, the main tool that we use  is a direct  corollary of theorem 1.12 in Hofer and Viterbo's proof of Weinstein conjecture in the presence of non-trivial holomorphic spheres \cite{hofer1992weinstein}. However the successful application of such method for the $n$-vortex problem on $\mathbb{S}^2_g$, to the contrary of that in the plane, turns out to be quite sensitive to the vorticity vector. For a detailed introduction to the $J$-holomorphic curves and symplectic geometry, see \cite{McduffJholomorphic}.

\subsection{Minimal Positive Action}
\label{Chapter_Symplectic_Subsection_Symplectic_Action}

Let $(V,\omega)$ be a closed symplectic manifold. For any smooth map $f: \mathbb{S}^2 \rightarrow V$, one can define the following integral for the pull-back form $f^{*}\omega$ on $\mathbb{S}^2$, i.e.,
\begin{align}
A_V^{\omega}(f) = \int_{\mathbb{S}^2} f^*\omega 
\end{align}
Since $V$ is a closed manifold, the number $A_V^{\omega}(f)$ is invariant in the same homotopy class $[f]\in \pi_2(V)$ (due to the Stokes theorem).  Define furthermore
\begin{align*}
m(V,\omega) &= \inf_{\substack{u \in \mathcal{C}^{\infty}(\mathbb{S}^2,V)\\ A_V^{\omega}(u) >0}}A_V^{\omega}(u) 
\end{align*}
One sets $m(V,\omega)= +\infty$ by convention if $\pi_2(V)=0$. The following lemma provides a situation where $m(V,\omega)$ can be calculated explicitly:

\begin{Lem}
\label{Lem_minimalsphere}
Let $\omega$ be a symplectic form on $\mathbb{S}^2$ and consider the product manifold $( (\mathbb{S}^2)^p, \omega_{\mathbf{\Gamma}})$, where 
\begin{align*}
&\mathbf{\Gamma}=(\Gamma_1,\Gamma_2,...,\Gamma_p)\in \mathbb({R}_{*})^n\\
&\omega_{\mathbf{\Gamma}} = \Gamma_1\omega \oplus \Gamma_2\omega \oplus...\oplus\Gamma_p\omega\\
\end{align*}
Then 
\begin{enumerate}
\item If $\Gamma_i \in \mathbb{Z}_{*}, 1\leq i \leq p$, then one has 
\begin{align}
m((\mathbb{S}^2)^p,\omega_{\mathbf{\Gamma}}) = h(\mathbf{\Gamma})\int_{\mathbb{S}^2}\omega
\end{align}
where $h(\mathbf{\Gamma})$ the highest common factor of $\{\Gamma_i\}_{1\leq i \leq p}$;
\item  If $\{\Gamma_i\}_{1\leq i \leq p}$ are not all commensurable, then 
$c(\mathbf{\Gamma}) = 0$.
\end{enumerate}
\end{Lem}
\begin{proof}
Let $f\in \mathcal{C}^{\infty}(\mathbb{S}^2, (\mathbb{S}^2)^p)$. The constraint $A_{\mathbb{S}^2}^{\omega}(f)>0$ excludes the possibility of $f$ being a constant. 
As for the product manifold $(\mathbb{S}^2)^p$, one has that
\begin{align}
\pi_2((\mathbb{S}^2)^p) = \pi_2(\mathbb{S}^2) \oplus  \pi_2(\mathbb{S}^2)\oplus...\oplus  \pi_2(\mathbb{S}^2) = \mathbb{Z}^p
\end{align}
For a given non-trivial homotopy class $[f] = (k_1,k_2,...,k_p)\in \pi_2((\mathbb{S}^2)^p)$, one sees that  
\begin{align}
m((\mathbb{S}^2)^p,\omega_{\mathbf{\Gamma}})
 = ([f] \cdot \mathbf{\Gamma})\int_{\mathbb{S}^2}\omega = (k_1\Gamma_1+ k_2 \Gamma_2 +...+ k_p \Gamma_p)\int_{\mathbb{S}^2}\omega
\end{align}
\begin{enumerate}
\item If $\Gamma_i \in \mathbb{Z}, 1\leq i \leq p$, then it follows by B\'ezout's identity in elementary number theory, that 
$m((\mathbb{S}^2)^p,\omega_{\mathbf{\Gamma}})
= h(\mathbf{\Gamma}) \int_{\mathbb{S}^2}\omega$.
\item If $\{\Gamma_i\}_{1\leq i \leq p}$ are not all commensurable, then we suppose without loss of generality that $\Gamma_1$ and $\Gamma_2$ are not commensurable, i.e. $\frac{\Gamma_1}{\Gamma_2} \notin \mathbb{Q}$. By considering the homotopy class of the type $[f]\in (k_1, k_2,0,0,...,0) \in \mathbb{Z}^p$, we have
\begin{align} 
\label{eq: commensurable}
m(\mathbb{S}^2)^p,\omega_{\mathbf{\Gamma}})\leq  \Gamma_1 \inf_{\substack{\mathbb{Z}\times\mathbb{Z}\setminus\{(0,0)\} \\k_1 + k_2 \frac{\Gamma_2}{\Gamma_1}>0}} (k_1 + k_2 \frac{\Gamma_2}{\Gamma_1}) c = 0
\end{align}
where the last equality is a consequence of standard Diophantine approximation \cite{schmidt1996diophantine}. 
\end{enumerate}
The lemma is thus proved.  
\end{proof}

\paragraph{}Let us still divide the vortex into two groups, one consists of a single vortex and the other are composed by the rest of vortices, i.e., 
\begin{description}
\item{Group I:} the $k^{th}$ vortex $\{{z}_k\}$;
\item{Group II:} the other $n-1$ vortices $\{z_1,z_2,...z_{k-1}, z_{k+1},...z_{n}\}$ 
\end{description}
 We then define the symplectic form:
\begin{align}
\omega_i = \Gamma_i \omega_g, \forall 1\leq i \leq n. \quad \omega_k^{*} = \bigoplus_{1\leq i \leq n, i\neq k}\omega_{i}
\end{align}
Thus $\omega_k$ denotes the symplectic form on $\mathbb{S}^2_g$ corresponding to the phase space of the $k^{th}$ vortex, and $\omega_k^{*}$ is the symplectic form on $(\mathbb{S}^2)^{n-1}$ consisting of all but the $k^{th}$ vortex. 
The above discussion motivates the following observation:
\begin{Pro}
\label{Proposition_Thin_HCF}
Given a vorticity vector $\mathbf{\Gamma}=(\Gamma_1,\Gamma_2,...\Gamma_n)$ for the $n$-vortex problem on $\mathbb{S}^2_g$ s.t. each vorticity is positive, i.e. $\forall 1\leq i\leq n,\text{ } \Gamma_i>0$. Then 
\begin{align}
m(\mathbb{S}^2_g,\omega_k) \leq m((\mathbb{S}^2_g)^{n-1},\omega_k^{*}) \Leftrightarrow \Gamma_k \text{ is thin with respect to } \mathbf{\Gamma}
\end{align}
\end{Pro}
\begin{proof}
This is a direct consequence of definition \ref{Definition_Thin_Vortex} and lemma \ref{Lem_minimalsphere}.
\end{proof}

In particular a direct corollary shows that for the identical $n$-vortex problem every vortex is a thin one:
\begin{Col}
Suppose that all the vorticities are identical, i.e. $\Gamma_i = \Gamma_j, \forall 1\leq i<j\leq n$. Then each vortex is a thin vortex.
\end{Col}
\begin{proof}
Clearly $\forall 1\leq k\leq n$, $\Gamma_k= \kappa(\Gamma_1, \Gamma_2,..., \Gamma_{k-1}, \Gamma_{k+1},...,\Gamma_k)$, the result follows.
\end{proof}

\subsection{Existence of Periodic Orbits}
\label{Chapter_Symplectic_Subsection_Periodic_Orbits}

In this subsection we study existence of periodic orbits on a product manifold of the type $\mathbb{S}^2 \times P$ based on the following lemma, which is proved in  \cite{hofer1992weinstein} by studying a specific module space of non-trivial pseudo-holomorphic spheres of a class $[f] \in \pi_2(\mathbb{S}^2 \times P)$ that has the minimal action among all the homotopy classes. 
\begin{Lem}[Theorem 1.12, \cite{hofer1992weinstein}]
\label{Theorem_Hofer_Viterbo}
Let $(P,\omega_P)$ be a compact symplectic manifold and $\omega$ a volume form on $\mathbb{S}^2$ such that 
\begin{align}
0 \leq \int_{\mathbb{S}^2}\omega \leq m(P,\omega_P)
\end{align}
Suppose $H: \mathbb{S}^2\times P  \rightarrow \mathbb{R}$ is a smooth Hamiltonian such that
\begin{align}
H|_{U_0} = h_0,\quad H|_{U_{\infty}} = h_{\infty}
\end{align}
where $U_{0}$ is a neighbourhood of some point $\mathbf{z}=(z_0,p_0)\in  \mathbb{S}^2 \times P $ and $U_{\infty}$ is a neighbourhood of the set $\{z_{\infty}\} \times P $, moreover $U_{\infty} \cap U_0= \emptyset $. Then the Hamiltonian system $\dot{z} = X_H(z)$  possesses a non-constant $T$-periodic solution $\mathbf{z}^{*}(t)$.
\end{Lem}
Our strategy is to let $((\mathbb{S}^2_g)^{n-1},\omega_k^{*})$ play the role of $(P,\omega_P)$ in lemma \ref{Theorem_Hofer_Viterbo}. Recall the fact that the closed characteristic on a hypersurface $S_c$ of a symplectic manifold does not depend on the Hamiltonian that defining it. Bearing this in mind one can prove the following theorem:
\begin{theorem}
\label{Thm_Main_tool}
Consider the two symplectic manifold $(\mathbb{S}^2_g,\sigma_1)$ and $((\mathbb{S}^2_g)^{n-1},\sigma_2)$. Let $S_c=H^{-1}(c)$ be a regular and compact hyper-surface of some Hamiltonian function $H: (\mathbb{S}^2)^{n}\rightarrow \mathbb{R}$. Consider the following conditions:
\begin{enumerate}
\item There holds the inequality \footnotemark{}
\begin{align} 
\label{Formula_symplectic_bound}
0 \leq \int_{\mathbb{S}_g^2}\sigma_1 \leq m((\mathbb{S}^2_g)^{n-1},\sigma_2) \tag{S1};
\end{align}
\item There exists a point $z_0\in \mathbb{S}_g^2$ s.t. 
\begin{align}
\label{Formula_Riemannian_separation}
S_c \cap (\{z_0\}\times (\mathbb{S}^2_g)^{n-1}) = \emptyset \tag{S2};
\end{align}
\end{enumerate}
Then there exists a sequence $c_i \rightarrow c$ s.t. each $S_{c_i}$ carries a non-constant periodic Hamiltonian trajectory.
\end{theorem}

\footnotetext{In corollary 1.13 Hofer and Viterbo has assumed that $\pi_{2}(P)=\{0\}$, in this case by convention $m(P,\omega_p)=\infty$ thus \textit{\`a fortiori} guarantees the validity of condition \eqref{Formula_symplectic_bound}}

\begin{proof}
First note that we can always suppose w.l.o.g. that $S_c$ is connected. To see this, suppose that $S_c$ has multiple components. We can take a component $\tilde{S}_c$ of $S_c$. There exists an open neighbourhood $V$ of $\tilde{S}_c$ disjoint from both other components of $S_c$ and $\{z_0\}\times (\mathbb{S}^2_g)^{n-1}$, and a smooth function $K:\mathbb{S}^2\times P\rightarrow \mathbb{R}$ s.t. 
\begin{align*}
    K = 
    \begin{cases}
      1 \text{ in } \tilde{S}_c;\\
      0 \text{ in } (\mathbb{S}^2\times P) \setminus V
    \end{cases}
\end{align*}
Let $\tilde{H}= H\cdot K$ and $\tilde{S}_c=\tilde{H}^{-1}(c)$ will be a compact and connected hyper-surface. 

Since the energy hyper-surface $S_c$ is compact and regular (i.e. $dH\neq 0$ on $S_c$), it is a $(2n-1)$-dimensional embedded submanifold. It follows by continuity that for small $\epsilon>0$, $U_{c,\epsilon}(H)$ is diffeomorphic to a $2n$-dimensional sub-manifold of $(\mathbb{S}^2_g)^{n}$ and $U_{c,\epsilon}(H)$ can be contracted to $S_c$.

First we show that $U_{c,\epsilon}(H)$ separates the manifold $(\mathbb{S}^2_g)^n$ into two parts. Since $H_1(\mathbb{S}^2, \mathbb{Z}_2) = 0$, by K\"unneth theorem, 
\begin{align*}
H_{*}((\mathbb{S}^2)^n, \mathbb{Z}_2) = H_{*}(\mathbb{S}^2, \mathbb{Z}_2) \otimes H_{*}(\mathbb{S}^2, \mathbb{Z}_2) \otimes ...\otimes H_{*}(\mathbb{S}^2, \mathbb{Z}_2) 
\end{align*}
Thus $H_1((\mathbb{S}^2)^n, \mathbb{Z}_2) = 0$. This implies that $(\mathbb{S}_g^2)^n \setminus U_{c,\epsilon}(H)$ has thus two components, according to the generalised Jordan-Brower theorem (\cite[corollary 8.8]{bredon2013topology}). The set $\{z_0\}\times (\mathbb{S}_g^2)^{n-1}$ is connected and will find itself in one component, namely $U_{\infty}$, and the other component is denoted by $U_0$.

The rest is the same as that for corollary 1.3 in \cite{hofer1992weinstein}: We choose a smooth function $\phi(s) : \mathbb{R} \rightarrow \mathbb{R}$ satisfying
\begin{align*}
\phi(s) =\begin{cases}
1, \quad\forall  s\geq \frac{1}{2}\epsilon\\
0, \quad\forall  s\leq \frac{-1}{2}\epsilon
\end{cases}\\
\frac{d\phi}{ds}>0, \quad \forall -\frac{\epsilon}{2} \leq s \leq  \frac{\epsilon}{2} 
\end{align*}
and construct the Hamiltonian as 
\begin{align*}
F(\mathbf{z}) = \begin{cases}
 \phi( H(z)-c), \quad \forall   \mathbf{z}\in U_{c,\epsilon} \\
0, \quad \forall  \mathbf{z}\in U_{0}\\
1, \quad \forall  \mathbf{z}\in U_{\infty}
 \end{cases}
\end{align*}
It turns out that $F(\mathbf{z})$ satisfies the assumption in theorem \ref{Theorem_Hofer_Viterbo}. Thus $F(\mathbf{z})$ possesses a periodic orbit $\mathbf{z}^{*}$ on $S_{c+\delta^*} \subset U_{c,\epsilon}$ for some $-\epsilon < \delta^{*} < \epsilon$, so does $H$.

Finally since $\epsilon$ is arbitrary, we can take $\epsilon_0 = \epsilon$ and $\displaystyle \epsilon_{n+1} = \frac{\epsilon_n}{2}$. The above argument then gives us a periodic orbit on $S_{c+\delta_n^{*}}$ for each $\delta^{*}_n$ as $\delta^{*}_n \rightarrow 0$. We conclude that the Hamiltonian $H$ possesses a sequence of periodic solutions $\mathbf{z}^*_{n}(t)$ with $H(\mathbf{z}^*_{n}(t)) \rightarrow c$.  
\end{proof}

Consider now the motion of $n$ vortices with positive vorticity on $\mathbb{S}^2_g$ equipped with an arbitrary Riemannian metric $g$. We basically only need to verify that theorem \ref{Thm_Main_tool} indeed applies. Let $\mathbf{z}=(z_1,z_2,...,z_n)\in \Sd$. Recall that when two vortices $z_i,z_j$ are approaching each other, the Green function $G(z_i,z_j)$ behaves like $-\log{d_g(z_i,z_j)}$ and goes to positive infinity. On the other hand since $\mathbb{S}^2$ is compact, $G(z_i,z_j)$ is bounded from below, so is the Robin mass function. To summarise, if one fixes $z_i$ and leaves all the other vortices $z_j, j\neq i$ to move freely, $H_g$ is bounded from below and always achieves its minimum, meanwhile it is never bounded from above. As a result one can define the following two values for each $1\leq i\leq n$:
\begin{align}
c_1(i, g) = \min_{\eta\in \mathbb{S}^2_g} \min_{ \substack{\mathbf{z} \in \Sd \\ z_i = \eta}    }  H_g(\mathbf{z}),\quad c_2(i, g) = \max_{\eta\in \mathbb{S}^2_g} \min_{ \substack{\mathbf{z} \in \Sd \\ z_i = \eta}    }  H_g(\mathbf{z})
\end{align}
Clearly $c_1(i,g)$ does not depend on $i$ and is the global minimum value of $H_g$. Moreover $c_1(i,g)\leq c_2(i,g)$ always holds. We remind the reader that there exists indeed cases where $c_1(i, g)= c_2(i, g)$, for instance when the $n$ identical vortices move on $\mathbb{S}^2_{g_0}$. In this case $c_1(i, g)= c_2(i, g)$ due to the $\mathbf{SO}(3)$ symmetry. However one sees that: 
\begin{Lem}
Fix $\mathbf{\Gamma} = (\Gamma_1,\Gamma_2,...,\Gamma_n) \in \mathbb{R}^n_{+}$. For $g(\rho)$ where $\rho \in \mathcal{D}_M$, $c_1(i,g) < c_2(i,g)$ strictly, $\forall 1 \leq i \leq n$.
\end{Lem}
\begin{proof}
Clearly by definition $c_1(i,g) \leq c_2(i,g)$. If $c_1(i,g) = c_2(i,g)$ for some $1\leq i \leq n$ and $g\in [g]$, it means that 
\begin{align*}
\forall \zeta\in \mathbb{S}^2_g,\quad  \min_{\eta\in \mathbb{S}^2_g} \min_{ \substack{\mathbf{z} \in \Sd \\ z_i = \eta}    }  H_g(\mathbf{z})=\min_{ \substack{\mathbf{z} \in \Sd \\ z_i = \zeta}    }  H_g(\mathbf{z}) = \max_{\eta\in \mathbb{S}^2_g} \min_{ \substack{\mathbf{z} \in \Sd \\ z_i = \eta}    }  H_g(\mathbf{z})
\end{align*}
In other words for any $\zeta \in \mathbb{S}^2_g$, one can find a configuration $\mathbf{z}_{\zeta}\in \Sd$ with $z_i = \zeta$ s.t.  $H_g(\mathbf{z})$ achieves its global minimum at $\mathbf{z}_{\zeta}$. As a result $\mathbf{z}_{\zeta}$ will be a critical point. Now since $\zeta\in \mathbb{S}^2_g$ can be an arbitrary point, we will find thus infinitely many fixed points as $\zeta$ goes over $\mathbb{S}^2_g$ . 

But for $g=g(\rho)$ with $\rho \in \mathcal{D}_M$ this is impossible due to theorem \ref{Theorem_Main_Morse}, since all the vorticities are positive thus condition \eqref{Condition_No_Collision} holds. The lemma is proved.  
\end{proof}

The above lemma implies that for certain energy, its energy hyper-surface satisfies some separation condition (see figure \ref{Fig_separation_condition}), which is the last ingredient towards the prove of theorem \ref{Theorem_Main_Periodic_Orbits}:

\paragraph{\textbf{Proof of Theorem \ref{Theorem_Main_Periodic_Orbits}}:}
Pick up a $g(\rho)$ with $\rho \in \mathcal{D}_{M}$, the Hamiltonian function $H_g$ becomes a Morse function, as a result $\forall 1\leq i\leq n,c_1(i,g) < c_2(i,g) $. Consider the thin vortex $z_k$ with vorticity $\Gamma_k$ and the phase space as the symplectic manifold $(\mathbb{S}^2_g \times (\mathbb{S}^2_g)^{n-1}, \omega_k \oplus \omega_k^{*}$). By proposition \ref{Proposition_Thin_HCF}, one has that 
\begin{align}
\label{Formula_symplectic_bound_spheres}
0 \leq \int_{\mathbb{S}_g^2}\omega_k \leq m((\mathbb{S}_g^2)^{n-1},\omega_k^{*}) 
\end{align}

Next take any $ c \in \mathbb{R}$ s.t. $c_1(k,g)<c<c_2(k,g)$. By the definition of $c_1(k,g)$ and $c_2(k,g)$, there exists a point $z_{c}\in \mathbb{S}^2$ s.t. 
\begin{align}
c <\min_{ \substack{\mathbf{z} \in \Sd \\ z_k = z_c}    }  H_g(\mathbf{z}) < c_2(k,g)
\end{align}

\begin{figure}[ht]
\begin{center}
\includegraphics[width=80mm,scale=0.5]{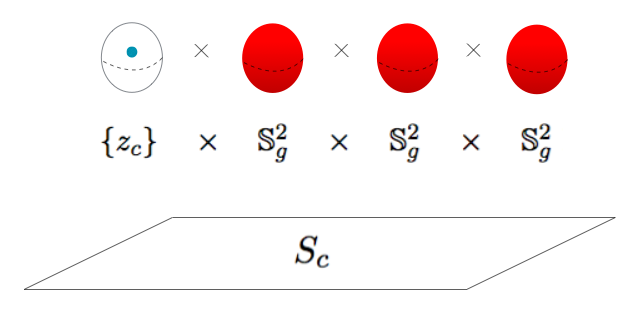}
\newline
\end{center}
\small
\textit{Suppose that $n=4$ with $z_1$ being a thin vortex, and that $g$ is s.t. $c_1(1,g)< c_2(1,g)$. Then for any $c\in\mathbb{R}$ s.t. $c_1(1,g)< c<c_2(1,g) $, there exists a point $z_c\in \mathbb{S}^2_g$ s.t. the whole set $\{z_c\}\times\mathbb{S}_g^2  \times \mathbb{S}_g^2   \times \mathbb{S}_g^2 $ locates on one side of the energy surface $S_c$ }
\caption{The Separation Condition}
\label{Fig_separation_condition}
\end{figure}

In other words (see figure \ref{Fig_separation_condition}), 
\begin{align}
S_c \cap (\{z_c\} \times (\mathbb{S}_g^2)^{n-1})  = \emptyset
\end{align}
The result follows by applying theorem \ref{Thm_Main_tool}. \QED

\subsection{Vortex Dipole and Contact Structure}
\label{Chapter_Symplectic_Subsection_Contact_Structure}

Theorem \ref{Theorem_Main_Periodic_Orbits} implies the existence of a periodic orbit near any given energy level $c_1(i,g)<c<c_2(i,g)$. But it is unknown whether on $S_c=H^{-1}_g(c)$ itself there exists one unless more structure, for instance the contact structure, could be added to $S_c$. If $S_c$ is of contact type, then it is \`a priori stable and together with theorem \ref{Theorem_Main_Periodic_Orbits} one will find a periodic orbit on $S_c$ itself. The existence of a periodic orbit on a regular and compact hyper-surface of contact type is known as the Weinstein's conjecture. It has first been proved by Viterbo \cite{viterbo1987proof} for $\mathbb{R}^{2n}$ and later on for various situations, see \cite{ginzburg2005weinstein} for a review on this subject.

The existence of a contact structure on $S_c$ can be shown by exhibiting a transversal Liouville vector field in the neighbourhood of $S_c$:
\begin{Pro}[\cite{weinstein1979hypotheses}]
\label{Proposition: Liouville Vector Field}
An energy hyper-surface $S_c\subset M$ is of contact type if and only if there exists a vector field $X$, defined on a neighbourhood $U$ of $S_c$, satisfying
\begin{enumerate}
\item $L_X\omega = \omega$ on $U$;
\item $X(z) \notin \mathcal{T}_zS_c, \forall z\in S_c$.
\end{enumerate} 
\end{Pro}
We restrict ourselves to the case of identical two vortices, i.e., $\Gamma_1=\Gamma_2=1$ to simplify our discussion. The $2$-vortex problem with general vorticity can be treated similarly. We can also suppose that the sphere is centered and has radius $1$.
\begin{figure}[ht]
\begin{center}
\includegraphics[width=120mm,scale=0.5]{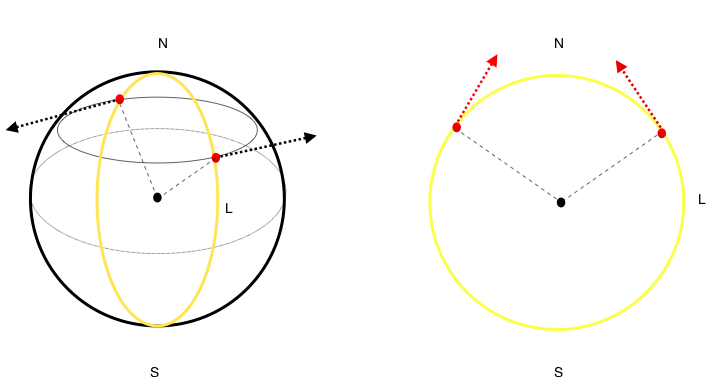}
\newline
\end{center}
\small
\textit{The yellow great circle is the Prime meridian. Two red points represent $z^{\alpha}_{1},z^{\alpha}_{2}$ respectively. Under the dynamics, $z^{\alpha}_{1},z^{\alpha}_{2}$ will rotate at uniform speed on the same latitude, represented by the black arrays on the left. The red arrays on the right represent the Liouville vector field at this configuration.}
\caption{Liouville Vector Field for Hyper-surface of $S^2_{g_0}$}
\label{Fig_contact_type}
\end{figure}
\paragraph{}First We fix artificially a north pole $N$, a south pole $S$ and a prime meridian (The Greenwich, in yellow in figure \ref{Fig_contact_type} ) $L$ on $\mathbb{S}^2$. For any $\alpha\in(c-\epsilon,c+\epsilon)$, there is a unique point $\mathbf{z}^{\alpha} =(z^{\alpha}_{1},z^{\alpha}_{2}) \in H^{-1}(\alpha)$ s.t. $z^{\alpha}_{1},z^{\alpha}_{2}\in L$ and the $N$ is on the middle point of the big arc connecting $z^{\alpha}_{1},z^{\alpha}_{2}$. Next for these $\mathbf{z}^{\alpha}$ on $L$ we assign the tangent vector field pointing to the north pole at each vortex (red arrays in figure \ref{Fig_contact_type} right), Such a tangent vectors tend to drag two vortices closer to each other along $L$ towards the north pole N, hence will increase the energy. As a result they must be transversal to the energy hypersurface $S_{\alpha}$. Finally for configurations $w\in S_{\alpha}$ on $S_{\alpha}$ other than this configuration, we can use elements of $\mathbf{SO}(3)$ to push the vector assigned on $\mathbf{z}_{\alpha}$ to a tangent vector of $w$. The detailed description of this idea is provided in appendix \ref{App_Chapter_Contact_Round_Sphere}.
\begin{Lem}
\label{Lem_contact_unit_sphere}
Consider the identical $2$-vortex on $\mathbb{S}^2_{g_0}$ with $\Gamma_1 = \Gamma_2=1$. Let 
\begin{align*}
c> \min_{z_1,z_2\in \mathbb{S}^{2}, z_1 \neq z_2}H_{g_0}=-\frac{1}{2\pi}log 2
\end{align*}Then the energy-surface $H^{-1}(c)$ is of contact type.
\end{Lem}
\begin{proof}
See appendix \ref{App_Chapter_Contact_Round_Sphere}.  
\end{proof}

Next we prove that the contact structure is stable if the $g_{0}$ has undergone a small change of metric. We only need to prove the case where the volume is preserved before and after the change of metric, thanks to the corollary \ref{Corollary_homothetic_change_of_metric}. Now due to the Dacorogna-Moser theorem, there exists a volume-preserving diffeomorphism $\phi: (\mathbb{S}^2, \omega_{g_0}) \rightarrow (\mathbb{S}^2,\omega_g)$. Moreover $dim(\mathbb{S}^2) = 2$, $\phi$ is a symplectomorphism. It follows that $\phi$ induces the symplectomorphism 
\begin{align}
\psi: \mathbb{S}_{g_0}^2\times \mathbb{S}_{g_0}^2 &\rightarrow \mathbb{S}_{g}^2\times \mathbb{S}_{g}^2 \notag\\
 (z_1,z_2) & \rightarrow (\phi(z_1), \phi(z_2))
\end{align}
Moreover, as $\rho \xrightarrow{\mathcal{C}_1} 0$, $\phi\xrightarrow{\mathcal{C}_1} Id$.

\begin{Pro}
\label{Proposition_contact_deformed_sphere}
Let $g_0$ be the standard round metric and $g$ be a Riemannian metric $g$ conformal to $g_0$, i.e. $g= e^{2\rho}g_0$ with $V_g(\mathbb{S}^2) = V_{g_0}(\mathbb{S}^2)$, and denote $S_c= H^{-1}_{g}(c)$. Then for $\norm{\rho}_{\mathcal{C}^1}$ small, if 
\begin{align}
 \forall z\in \mathbb{S}^2,\quad    (\phi(z), \phi(-z)) \notin S_c  
\end{align}
Then $S_c$ is of contact type.
\end{Pro}

\paragraph{\textbf{Proof of theorem \ref{Theorem_Main_Contact}}:}
Under the assumption, there exists $\tilde{c}> -\frac{1}{2\pi}log 2, \epsilon>0$ s.t. that $\tilde{c}-\epsilon>  -\frac{1}{2\pi}log 2$ and $\psi^{-1}(S_c)\subset U_{\tilde{c},\epsilon}(H_{g_{0}})$. Now according to lemma \ref{Lem_contact_unit_sphere}, there is a transversal Liouville vector field $\mathbf{v}$ in $U_{\tilde{c},\epsilon}(H_{g_{0}})$. On one hand, the transversality is an open property hence stays valid for $\rho$ small enough; on the other hand, $\psi$ is a symplectomorphism thus its push-forward sends Liouville vector field to Liouville vector field. Together one conclude that $\mathbf{v}_{\psi} = \psi_{*}\mathbf{v}$ is also a transversal Liouville vector field in the neighbourhood $\psi(U_{\tilde{c},\epsilon}(H_{g_{0}}))$ of $S_c$, which implies that $S_c$ is also of contact type.  \QED

\paragraph{}As an immediate application of the theorem, we see that 
\begin{Col}
Suppose that $I_g$ is as in theorem \ref{Theorem_Main_Periodic_Orbits}, and $\psi$ as in theorem \ref{Theorem_Main_Contact}. For $\norm{\rho}_{\mathcal{C}^1}$ small, if   
\begin{align}
 h_g =  \{h \hspace{0.25em}|\hspace{0.25em} h = H_g(\phi(z), \phi(-z)), z\in \mathbb{S}^2 \} 
\end{align}
Then for each $c\in I_g \setminus h_g$, $S_c= H_g^{-1}(c)$ possesses at least a periodic orbit on it. 
\end{Col}
\begin{proof}
This results from theorem \ref{Theorem_Main_Periodic_Orbits} and theorem \ref{Theorem_Main_Contact} and the discussion in the beginning of this sub-section.  
\end{proof}
\subsection{Exclusion for Perverse Choreography}

\label{Chapter_Symplectic_Subsection_Perverse_Choreography}
The method discussed in sub-section \ref{Chapter_Symplectic_Subsection_Periodic_Orbits} can be applied to find periodic orbits for the identical $n$-vortex problem on $\mathbb{S}^2_g$ with $g=g(\rho), \rho \in \mathcal{D}_{M}$. However we don't know the precise configuration of such an orbit. Let $\mathbf{z}= (z_1, z_2,...,z_n)$ with $z_i=(x_i,y_i)$ be the Clebsch variables for the motion of $n$ vortices moving on $\mathbb{S}^2_g$. Then as the case in the plane \cite{MarsdenCoadjoint}, the gauge group for the Poisson map 
\begin{align}
\mathbf{z}=(z_1,z_2,...,z_n) \rightarrow \sum_{i=1}^{n}\Gamma_i \delta_{(x_i,y_i)} \omega_g 
\end{align}
is trivial unless some $\Gamma_i$'s are identical. Clearly this holds regardless of the Riemannian metric tensor assigned to $\mathbb{S}^2$. In particular, if all the $\Gamma_i$'s are all identical, one might hope that there exist choreographies, i.e., periodic orbits with the property that 
\begin{align}
z_i(t-\frac{T}{n}) = z_{i+1}(t), 1\leq i\leq n 
\end{align}
where by convention $z_{0}= z_n$. In other words, if we introduce the transformation 
\begin{align}
    \mathcal{P}: \Sd &\rightarrow \Sd, \notag\\
    \mathbf{z}=(z_1,...z_n) &\rightarrow \mathcal{P}\mathbf{z}=(z_n,z_1,...z_{n-1})
\end{align}
Then $\mathbf{z}$ being a choreography is equivalent to $\mathcal{P}\mathbf{z}(t)=\mathbf{z}(t+\frac{T}{n})$. It is in still open whether choreographies exist for $n$-vortex problem on $\mathbb{S}^2_g$ with general $g$, although on $\mathbb{S}^2_{g_0}$ they have been found \cite{tronin2006absolute} \cite{carlos2018choreography}.

On the other hand, following the discovery of such symmetric solution in the $n$-body problem \cite{chenciner2000remarkable}, Chenciner has asked the question about whether there exists perverse choreography, i.e., choreography with non-identical mass (or vorticities), see \cite{ChencinerPerverse}. One can also consider the possibility for the existence of a perverse choreography in the $n$-vortex problem. In \cite{Celli2003On}, Martin Celli has shown that the $n$-vortex problem on the plane does not permit perverse choreography. We would like to study the problem on the Riemannian manifold $\mathbb{S}^2_g$. First we prove the following proposition using an idea from \cite{ChencinerPerverse}. 
\begin{Pro}
\label{Proposition_Fixed_Points_Adjustified_Vorticity}
Consider the $n$-vortex problem on $\mathbb{S}^2_g$. Suppose $\mathbf{z}(t)$ is a choreography with period $T>0$ and with vorticity vector $\mathbf{\Gamma}= (\Gamma_1,\Gamma_2,...,\Gamma_n)$ not all identical. Then any configuration on this choreography will be a stationary configuration for the $n$-vortex problem on $\mathbb{S}^2_g$ with vorticity vector $\mathbf{\tilde{\Gamma}}= (\Gamma_1-\bar{\Gamma},\Gamma_2-\bar{\Gamma},...,\Gamma_n-\bar{\Gamma})$.
\end{Pro}
\begin{proof}
Let's consider the three Hamiltonians 
\begin{align}
H_g(\mathbf{z}) &= \sum_{1\leq i <j \leq n } \Gamma_i\Gamma_j G_g(z_i,z_j) + \sum_{1\leq i \leq n} \Gamma_i^2 R_g(z_i) \\
\bar{H}_g(\mathbf{z}) &= \sum_{1\leq i <j \leq n } \bar{\Gamma}^2 G_g(z_i,z_j) + \sum_{1\leq i \leq n} \bar{\Gamma}^2 R_g(z_i) \\
\tilde{H}_g(\mathbf{z}) &= \sum_{1\leq i <j \leq n } (\Gamma_i-\bar{\Gamma})(\Gamma_j-\bar{\Gamma})G_g(z_i,z_j) + \sum_{1\leq i \leq n} (\Gamma_i-\bar{\Gamma})^2 R_g(z_i)
\end{align}

We only need to verify the claim for those vorticity s.t. $\Gamma_i \neq \bar{\Gamma}$. To this end we will directly look at the intrinsic definition of Hamiltonian vector field. Let $X_{H_g}=(X^1_{H_g},X^2_{H_g},...,X^n_{H_g})$. Suppose that $\Gamma_1\neq \bar{\Gamma}$ without loss of generality. By definition of Hamiltonian system,
\begin{align*}
\forall s\in \mathbb{R},1\leq k\leq n, \quad     i_{X_{H_g}^k} \omega_k (\mathbf{z}(s))= d_{z_k} H(\mathbf{z}(s))
\end{align*}
As before let $\omega_i= \Gamma_i \omega_g$ for $1\leq i\leq n$. Since $\mathbf{z}$ is a choreography, one sees that 
\begin{align}
\Gamma_k i_{X_{H_g}^1}\omega_g (\mathbf{z}(s))&=i_{X_{H_g}^1} \omega_k (\mathbf{z}(s))=i_{X_{H_g}^k} \omega_k (\mathbf{z}(s+\frac{(k-1)T}{n}))\notag\\ 
&= d_{z_k} H(\mathbf{z}(s+\frac{(k-1)T}{n}) 
= d_{z_1} H(\mathcal{P}^{k-1}\mathbf{z}(s))
\end{align}
This implies that 
\begin{align*}
 \sum_{1\leq k\leq n }i_{X_{H_g}^1}\omega_g (\mathbf{z}(s)) = \sum_{1\leq k\leq n }\frac{1}{\Gamma_k}d_{z_1} H(\mathcal{P}^{k-1}\mathbf{z}(s))
\end{align*}
Hence
\begin{align}
i_{X_{H_g}^1}\bar{\Gamma}\omega_g = \frac{\bar{\Gamma}}{n}\sum_{1\leq k\leq n }\frac{1}{\Gamma_k}d_{z_1}H(\mathcal{P}^{k-1}\mathbf{z}(s)) = d_{z_1}\bar{H}(\mathbf{z})
\end{align}
Finally consider the $n$-vortex problem with vorticity $\tilde{\mathbf{\Gamma}}$. The equation is 
\begin{align}
i_{X_{\tilde{H}_g}^1}\tilde{\Gamma}_1\omega_g = d_{z_1}\tilde{H}_g
\end{align}
By assumption $\tilde{\Gamma}_1 \neq 0$, hence 
\begin{align}
i_{X_{\tilde{H}_g}^1}\omega_g = \frac{1}{\tilde{\Gamma}_1}d_{z_1}\tilde{H}_g = \frac{1}{\Gamma_1 -\bar{\Gamma}}d_{z_1}\tilde{H}_g = \frac{1}{\Gamma_1} d_{z_1}H_g- \frac{1}{\bar{\Gamma}}d_{z_1}\bar{H}_g =i_{X_{H_g}^1}\omega_g-i_{X_{H_g}^1}\omega_g = \mathbf{0}
\end{align}
Hence $X_{\tilde{H}_g}^1 = \mathbf{0}$, similar for other $z_i$ with $\Gamma_i \neq \bar{\Gamma}$.
\end{proof}

Now we are ready to exclude the perverse choreography for $\mathbb{S}^2_{g(\rho)}$ with $\rho \in \mathcal{D}_M$, as is shown by the following lemma:
\paragraph{\textbf{Proof of theorem \ref{Theorem_Main_No_Perverse_Choregraphy}}:}

Consider the $n$-vortex problem on $\mathbb{S}^2_g$. Suppose $\mathbf{z}(t)$ is a perverse choreography with period $T>0$ and with vorticity vector $\mathbf{\Gamma}= (\Gamma_1,\Gamma_2,...,\Gamma_n)$. By proposition \ref{Proposition_Fixed_Points_Adjustified_Vorticity} this means that any configuration on this orbit will be a fixed point for the vorticity vector $\tilde{\mathbf{\Gamma}}= (\Gamma_1-\bar{\Gamma},\Gamma_2-\bar{\Gamma},...,\Gamma_n-\bar{\Gamma})$. In other words, by moving along the choreography, we have achieved a continuous family of fixed points for the $n$-vortex problem $\mathbb{S}^2_g$ with vorticity vector $\tilde{\mathbf{\Gamma}}$. However 
\begin{align*}
\sum_{1\leq i<j\leq n} ( \Gamma_{i}-\bar{\Gamma})(\Gamma_{j}-\bar{\Gamma}) &= \frac{1}{2} \big( (\sum_{1\leq p \leq n}(\Gamma_{i}-\bar{\Gamma}))^2- \sum_{1\leq i \leq n} (\Gamma_{i}-\bar{\Gamma})^2  \big )\\
&= - \sum_{1\leq i \leq n} (\Gamma_{i}-\bar{\Gamma})^2<0
\end{align*}
which according to theorem \ref{Theorem_Main_Morse} leads to a contradiction for $g=e^{2\rho}g_0$ with $\rho \in \mathcal{D}_{M}$.  \QED


\section{Some Further Remarks}
\label{Chapter_Further_Remarks}

Finally we discuss some further directions, for example the possibility and difficulty to generalise some of the results presented in above sections to other Riemann surfaces and to the case of mixed vorticities.


\subsection{Other Riemann Surfaces}
\label{Chapter_Remarks_General_Riemann_Surface}
\paragraph{}It is plausible that some of the points discussed in this article could be applied to more general Riemann surfaces, as long as one perturbs the metric merely in its conformal class. 
Another possibility is to add boundary, instead of working on closed surfaces. We refer to \cite{lin1941motionI} \cite{gustafsson1979vortex} \cite{crowdy2005motion} \cite{sakajo2009equation} which present the formulation given different boundaries, both for simply connected and multiple connected domains. The boundary could also be considered as a parameter to produce the Morse property generically, thus it is possible to generalise the results of Bartsch et al. \cite{bartsch2017morse} from the planar case to be curved case even without changing the Riemannian metric. With the presence of the boundary, the topology of closed surface might no longer count any importance since such regions are homeomorphic to (either simply or multiply connected) domains in $\mathbb{R}^n$, but one then needs to handle the new singularities in the Hamiltonian function fulfilling the boundary condition, see \cite{bartsch2016periodic,bartsch2017global,gebhard2018periodic,gebhard2019stability} for the study of periodic orbits with general boundary in $\mathbb{R}^2$. On the other hand, it might be interesting to compare the orbits bifurcating from either fixed points or periodic orbits, since there are evidences \cite{boatto2008curvature} that the stability of orbits might be reinforced or weakened by the geometry of the surface.

\subsection{Mixed Vorticity}
\label{Chapter_Remarks_Mixed_Vorticity}

The search of periodic orbit for vorticity vectors with mixed signs will be more complicated, and the method discussed above cannot be applied directly, due to various reasons. 
\paragraph{Unbounded Hamiltonian Function}
When more than two vortices of different signs are considered, the Hamiltonian is neither bounded from above nor from below. As a result we will no longer have the separation condition in theorem \ref{Thm_Main_tool}.

\paragraph{Non-compact Energy Hypersurface} Moreover, on fixed level set, there exists configurations with arbitrarily small distance between multiple vortices of different signs, thus the energy surface is no longer compact. Actually one do not even know whether the flow is global in general, as collision might indeed happen in finite time. Nevertheless, for the 3-vortex problem with vanishing total vorticity, the flow is still global regardless of the particular Riemannian metric:
\begin{Pro}
Consider the 3-vortex problem on $\mathbb{S}^2_g$ for any $g$. If $\Gamma_1+\Gamma_2 +\Gamma_3 = 0$, then the solution never blows up.
\end{Pro}
\begin{proof}
Since $\mathbb{S}^2$ is closed manifold, the only possibility for blowing up are the collisions. As there is only one conformal equivalent class, any metric $g$ is conformal to the round metric $g_0$. In case that $\Gamma_1 + \Gamma_2 + \Gamma_3=0$, one has that 
\begin{align}
\label{Formula_Hamiltonian_Vortex_Dipole_under_a_Conformal_Change_of_Metric}
H_{g}(\mathbf{z}) = H_{g_0}(\mathbf{z}) -\frac{1}{2\pi}\sum_{i=1}^3\Gamma_i^2\rho(z_i)
\end{align}
The function $H_{g_0}(\mathbf{z}) \rightarrow\infty$ were there any collision, either pairwise or triple \cite[Theorem 2]{sakajo1999motion}. But $\rho$ is a regular function hence it never blows up. As a result, $H_{g_0}(\mathbf{z}) \rightarrow \infty$ implies that $H_{g}(\mathbf{z}) \rightarrow \infty$, which is impossible because the $H_g$ is a conserved quantity.
\end{proof}

One possibility to overcome such difficulty is to restrict ourselves to the case where some symmetric condition on vorticities and also the metirc $g$ should be posed, as an attempt to reduce the degree of freedom.

\subsection{The Conditions \eqref{Condition_No_Collision} and \eqref{Condition_Thin_Vorticity} on Vorticity}
\label{Chapter_Remarks_Conditions_Vorticity}

In this article we have focused on the perturbation of Riemannian metric, instead of changing the vorticity. It might be possible to prove the Morse property under generic perturbation of vorticity vector. However our consideration is that, while condition \eqref{Condition_No_Collision} is an open property and is stable under small perturbation of vorticity, the condition \eqref{Condition_Thin_Vorticity} is indeed not an open property, and even small adjustment of vorticity might break this condition. To be precise, define
\begin{align}
\mathcal{Q}_n = \{\mathbf{\Gamma} \in (\mathbb{R}_{+})^{n}, \mathbf{\Gamma} \text{ possesses a thin vorticity}\}
\end{align}
and let $\mu_{n}$ be the Lebesgues measure on $\mathbb{R}^{n}$, then it turns out that 
\begin{Pro}
$\mathcal{Q}_2 = \mathbb{R}_{+}^2 $ but $\mu_n(\mathcal{Q}_n) = 0 $ for $n\geq 3$.
\end{Pro}
\begin{proof}
$\mathcal{Q}_2 = \mathbb{R}_{+}^2 $ is a direct consequence of lemma \ref{Lem_minimalsphere}. Actually, $\forall \mathbf{\Gamma}=(\Gamma_1, \Gamma_2)$, we can w.l.o.g. suppose that $0<\Gamma_1 < \Gamma_2$ and $z_1$ will be the thin vortex.  Now for $n\geq 3$. we set $\Gamma_1 = a>0$ and study the set 
\begin{align}
\mathcal{Q}(a) = \{(\Gamma_2, \Gamma_3,...,\Gamma_n) \in (\mathbb{R}^{+}_{*})^{n-1}, m(M,a\omega_g) \leq m(M^{n-1},\omega_1^{*}) \}
\end{align}
Clearly $\mu_{n-1}(\mathcal{Q}(a)) = 0$. To finish the theorem one needs only to apply the Fubini's theorem.
\end{proof}
As a result, for $n\geq 3$, the vorticity vectors that fulfil condition \eqref{Condition_No_Collision} is a full measure set, while the vorticity vectors that fulfil condition \eqref{Condition_Thin_Vorticity} is a null set.
\appendix
\section{Linearization for $dH_g$ }
\label{App_Chapter_Linearised_differential}
In this appendix we present details of the linearised operator used in chapter \ref{Chapter_Riemannian_Metric_and_Hamiltonian_Function}.
Let $H_g$ be of the $n$-vortex problem on $\mathbb{S}^2_g$, and $g(\rho) = e^{2\rho}g_0$ for $\rho \in \mathcal{C}^{2}(\mathbb{S}^2, \mathbb{R})$. We define the map 
\begin{align}
f:\Se \times \mathcal{C}^{2}(\mathbb{S}^2, \mathbb{R})
\rightarrow  \mathcal{T}^{*}\Se \notag \\
(\mathbf{z}, \rho) \xrightarrow{f} dH_{g(\rho)}(\mathbf{z}) \tag{F}
\end{align}
Note that $f$ is a map between two Banach manifold ($\mathcal{C}^{2}(\mathbb{S}^2, \mathbb{R})$ is itself a Banach space). We show that $f$ is of class $\mathcal{C}^1$ by calculating its Fr\'echet derivative $df$. To this end, take a coordinate patch $(U,\phi)$ s.t.
$\mathbf{z}\in U\subset \Se $ and 
\begin{align*}
    \phi: U &\rightarrow V \subset \mathbb{R}^{2n}\\
    \phi(\mathbf{z})&=(\x,\y)= (x_1,y_1, x_2,y_2,...x_{n},y_{n})
\end{align*}
is a local coordinate of $\Se$, and consequently 
\begin{align*}
(\x,\y,d\x,d\y)= (x_1,y_1,...,x_{n},y_{n}, dx_1,dy_1,...,dx_{2n},dy_{2n})
\end{align*}
forms a local coordinate on a neighborhood of $f(\mathbf{z}, \rho)\in \mathcal{T}^{*}\Se$.  Consider a smooth curve 
\begin{align*}
\gamma(t)= \big(\alpha(t)=(\alpha_1(t),...\alpha_n(t)), \beta(t) \big): (-1,1) &\rightarrow  U \times    \mathcal{C}^{2}(\mathbb{S}^2, \mathbb{R})\\
\alpha(0) = \z, \alpha'(0)= \mathbf{w}; &\quad \beta(t)= \rho+ t\psi 
\end{align*}
according to formula \eqref{Formula_Hamiltonian_under_a_Conformal_Change_of_Metric} we see that
\begin{align*}
f(\gamma(t))
&= dH_{g(\beta(t))}(\alpha(t))   \notag\\
&=d\big(H_{g(\rho)}(\alpha(t))+\frac{1}{2\pi}\sum_{i=1}^n \Gamma_i^2 t\psi(\alpha_i(t)) -\frac{\sum_{i=1}^{n}
\Gamma_i}{V_{g(\beta(t))}(\mathbb{S}^2)  } \sum_{i=1}^n \Gamma_i \Delta_{g(\rho)}^{-1}e^{2t\psi(\alpha_i(t))}   \big)\\
&=dH_{g(\rho)}(\alpha(t))+\frac{1}{2\pi}\sum_{i=1}^n \Gamma_i^2 dt\psi(\alpha_i(t)) +\frac{\sum_{i=1}^{n}
\Gamma_i}{V_{g(\beta(t))}(\mathbb{S}^2)  } \sum_{i=1}^n \Gamma_i \Delta^{-1}de^{2t\psi(\alpha_i(t))}  
\end{align*}
where $\Delta = (\delta d + d \delta)$ is the Laplace-de Rham operator on $\mathbb{S}^2_{g(\rho)}$, and when acting on scalar functions $h:\mathbb{S}^2\rightarrow \mathbb{R}$ we have $\Delta h = -\Delta_{g(\rho)} h$. Moreover we have used the fact that $\Delta^{-1}d = d\Delta^{-1}$.
Denote $\hat{H}_g= H_g\circ \phi^{-1} $ and $\hat{\psi} = \psi \circ \phi^{-1}$, one calculates in local coordinate that 
\begin{align*}
\frac{df(\gamma(t))}{dt} =& \big((\dot{\x},\dot{\y}), d^2 \hat{H}_{g(\rho)}(\x,\y) (\dot{\x},\dot{\y})^T + \frac{1}{2\pi} \sum_{i=1}^n \Gamma_i^2 ( d\hat{\psi}(x_i,y_i)  +t d^2 \hat{\psi}(x_i,y_i) (\dot{x_i},\dot{y_i})^T  )\notag\\
&+\frac{\sum_{i=1}^{n}\Gamma_i}{V^2_{g(\beta(t))}} \int_{\mathbb{S}^2} e^{t\psi} \psi w_{g(\rho)} d(\sum_{i=1}^n \Gamma_i \Delta^{-1}e^{2t\hat{\psi}(x_i,y_i)}) \notag\\
&+\frac{\sum_{i=1}^{n}\Gamma_i}{V_{g(\beta(t))}(\mathbb{S}^2)} \sum_{i=1}^n \Gamma_i \Delta^{-1}\big( e^{2t\hat{\psi}(x_i,y_i)} (2t  \frac{d \hat{\psi}(x_i,y_i)}{dt}  +  2d\hat{\psi}(x_i,y_i)) \\
&+ 2t e^{2t\hat{\psi}(x_i,y_i)}  d^2\hat{\psi}(x_i,y_i) (\dot{x}_i,\dot{y_i})^T\big) \big)
\end{align*} 
Take $t=0$ we see that 
\begin{align*}
     \frac{df(\gamma(t))}{dt}|_{t=0} &= \big((\dot{\x}(0),\dot{\y}(0)),   d^2 \hat{H}_{g(\rho)}(\x(0),\y(0)) \mathbf{w}^T+\frac{1}{2\pi} \sum_{i=1}^n \Gamma_i^2  d\hat{\psi}(x_i(0),y_i(0)) \notag\\
     &+2\frac{\sum_{i=1}^{n}\Gamma_i}{V_{g(\beta(t))}(\mathbb{S}^2)} \sum_{i=1}^n \Gamma_i \Delta^{-1} d\hat{\psi}(x_i(0),y_i(0)) \big)
\end{align*}
In particular, if $\mathbf{z}$ is a critical point of $H_g(\rho)$, then the Hessian does not depend on the metric and we have 
\begin{align*}
     Df_{(\z,\rho)} (\mathbf{w},\psi)&=  \frac{df(\gamma(t))}{dt}|_{t=0} \notag\\
     &= \big (\mathbf{w}, d^2 H_{g(\rho)}(\z)(\mathbf{w})+\frac{1}{2\pi} \sum_{i=1}^n \Gamma_i^2  d\psi (z_i)  
     +2\frac{\sum_{i=1}^{n}\Gamma_i}{V_{g(\beta(t))}(\mathbb{S}^2)} \sum_{i=1}^n \Gamma_i \Delta^{-1} d\psi(z_i) \big )  \\
     &= \big ( \mathbf{w}, d^2 H_{g(\rho)}(\z)(\mathbf{w})+\frac{1}{2\pi} \sum_{i=1}^n \Gamma_i^2  d\psi (z_i)  
     -2d \big (\frac{\sum_{i=1}^{n}\Gamma_i}{V_{g(\beta(t))}(\mathbb{S}^2)} \sum_{i=1}^n \Gamma_i \Delta_{g(\rho)}^{-1} \psi(z_i) \big ) \big )
\end{align*}
The linearised map above consists of two parts, the first part $d^2 H_{g(\rho)}(\z)(\mathbf{w})$ is the Green function and the rest part is the Robin mass. Since all $\mathbf{z}\in \Se$ are uniformly separated from the collision set, moreover the Robin mass is a regular function, classical regularity estimate for elliptic operator then implies that $Df$ is continuous. 

\section{Contact Structure for Energy Surface on $\mathbb{S}^2_{g_0}$} 
\label{App_Chapter_Contact_Round_Sphere}
In this appendix we verify in details that the vector field suggested in section \ref{Chapter_Symplectic_Subsection_Contact_Structure} indeed is a transversal Liouville vector field in the neighbourhood of any regular hyper-surface $S_c= H^{-1}_{g_0}$. We will drop the index $g_0$ to simplify the notation without introducing any unexpected ambiguity, as we will only work with the round metric in this appendix. Here is a detailed proof of lemma \ref{Lem_contact_unit_sphere}:
\begin{proof}
Let the two poles $N$ and $S$ and the prime meridian $L$ be fixed as in section \ref{Chapter_Symplectic_Subsection_Contact_Structure}. First, note that the only fixed points for the 2-vortex problem are anti-polar configurations, on which the Hamiltonian achieves its minimum $-\frac{1}{2\pi}log 2$. As a result, $c$ is a regular value of $H$ and $U_{c,\epsilon}(H)$ is clearly a neighbourhood of $H^{-1}(c)$. 

For any $\alpha\in(c-\epsilon,c+\epsilon)$, there is a unique point $\mathbf{z}^{\alpha} =(z^{\alpha}_{1},z^{\alpha}_{2}) \in H^{-1}(\alpha)$ s.t. $z^{\alpha}_{1},z^{\alpha}_{2}\in L$ and the $N$ is on the middle point of the big arc connecting $z^{\alpha}_{1},z^{\alpha}_{2}$. In a local coordinate chart it is represented by 
\begin{align}
\mathbf{z}^{\alpha}= (p^{\alpha}_{1},q^{\alpha}_{1},p^{\alpha}_{2},q^{\alpha}_{2})=(\cos \theta_\alpha, 0, \cos\theta_\alpha, \pi)\in S_{\alpha}, \text{ with } \theta_{\alpha} = \arccos \frac{e^{-2\pi \alpha}}{2}
\end{align}
For $\mathbf{z}^{\alpha}$ we associate the vector (see figure \ref{Fig_contact_type} right).
\begin{align*}
\mathbf{v}(\mathbf{z}^{\alpha}) = (p^{\alpha}_{1},0,p^{\alpha}_{2},0) =  (\cos \theta_{\alpha}, 0, \cos\theta_{\alpha}, 0)
\end{align*}
Now for $\mathbf{w}^{\alpha}= (w^\alpha_1,w^\alpha_2)\in H^{-1}(\alpha), \mathbf{w}^{\alpha} \neq \mathbf{z}^{\alpha}$, we see from the Hamiltonian \eqref{Formula_Vortex_Round_Sphere} that $\norm{w^{\alpha}_1-w^{\alpha}_2} = \norm{z^{\alpha}_{1}-z^{\alpha}_{2}}< 2$. As a result, there is a unique element $\mathfrak{g}_{\mathbf{w}^{\alpha}}\in SO(3)$ s.t. $\mathfrak{g}_{\mathbf{w}^{\alpha}}(\mathbf{z}^{\alpha}) = \mathbf{w}^{\alpha}$. Note that here the action is the diagonal action on $\mathbb{S}^2\times \mathbb{S}^2$. We associate to $\mathbf{w}^{\alpha}$ the push-forward vector $\mathbf{v}(\mathbf{w}^{\alpha}) = (\mathfrak{g}_{\mathbf{w}^{\alpha}})_{*}  \mathbf{v}(\mathbf{z}^{\alpha})$. 

We verify the vector field $\mathbf{v}$ thus constructed is indeed a Liouville vector field in $U_{c,\epsilon}$.
First, $\mathbf{SO}(3)$ is isometric, hence
it turns out that for $\mathbf{z}^{\alpha}\in H^{-1}(\alpha)$:
\begin{align}
dH(\mathfrak{g}_{\mathbf{z}^{\alpha}}(\mathbf{z}^{\alpha}))(\mathfrak{g}_{\mathbf{z}^{\alpha}})_{*}  \mathbf{v}(\mathbf{z}^{\alpha})= dH(\mathbf{z}^{\alpha})(\mathbf{v}(\mathbf{z}^{\alpha}))>0
\end{align}
Hence $\mathbf{v}$ is transversal. 

Next we show that $\mathbf{v}$ is homothetic. On $L$ one has  $\mathbf{v}(\mathbf{z}) = \mathbf{v}(\mathbf{p},\mathbf{q})  = (\mathbf{p},\mathbf{0})$, thus  $L_{\mathbf{v}}\Omega = \Omega$ on $L$. Now out of $L$, let $\phi_\mathbf{v}^t$ be the flow of $\mathbf{v}$, then
\begin{align}
\label{Liouville_vector_field} 
\mathfrak{g}_{\mathbf{w}^{\alpha}}^{*}\mathcal{L}_{\mathbf{v}}\Omega_{}(\mathbf{w}^{\alpha}) 
&= \mathfrak{g}_{\mathbf{w}^{\alpha}}^{*}(i_{\mathbf{v}}d\Omega_{}(\mathbf{w}^{\alpha})  + d i_{\mathbf{v}}\Omega(\mathbf{w}^{\alpha})) \tag{Cartan formula}\\
&=\mathfrak{g}_{\mathbf{w}^{\alpha}}^{*} \circ d\circ i_{\mathbf{v}}\Omega (\mathbf{w}^{\alpha})  \notag\\
&= d \circ \mathfrak{g}_{\mathbf{w}^{\alpha}}^{*} \circ i_{(\mathfrak{g}_{\mathbf{w}^{\alpha}})_{*}  \mathbf{v}(\mathbf{z}^{\alpha})}\Omega (\mathbf{w}^{\alpha}) \notag \\
&= d \circ   i_{\mathbf{z}^{\alpha}} \mathfrak{g}_{\mathbf{w}^{\alpha}}^{*}\Omega(\mathbf{w}^{\alpha}) \tag{$\mathfrak{g}_{\mathbf{w}^{\alpha}}$ is symplectic on $\mathbb{S}^2\times \mathbb{S}^2$}\\
&= d\circ i_{\mathbf{v}(\mathbf{z}^{\alpha})} \Omega(\mathbf{z}^{\alpha}) \notag\\
&=  \mathcal{L}_{\mathbf{v}(\mathbf{z}^{\alpha})}\Omega(\mathbf{z}^{\alpha})\notag \\
&= \Omega(\mathbf{z}^{\alpha}) \tag{$\mathbf{z}^{\alpha}\in L$ }\\
&=\mathfrak{g}_{\mathbf{z}^{\alpha}}^{*}\Omega(\mathbf{\omega}^{\alpha}) \tag{$\mathfrak{g}_{\mathbf{z}^{\alpha}}$ is symplectic on $\mathbb{S}^2\times \mathbb{S}^2$}\\
\Rightarrow  \mathcal{L}_{\mathbf{v}}\Omega_{}(\mathbf{w}^{\alpha}) = \Omega(\mathbf{w}^{\alpha}) 
\end{align}

The lemma is proved.
\end{proof}
\paragraph{Acknowledgements}
The author is indebted to Jacques F\'ejoz, Eric S\'er\'e and Ke Zhang for inspiring discussions. The author also appreciates Universit\"{a}t Augsburg and Kyoto University for the hospitality during his stay, where part of this work is carried out. This project is supported by National Natural Science Foundation of China grant \textbf{11901160}.


\end{document}